\documentclass[a4paper,11pt]{article}

\usepackage[utf8]{inputenc}
\usepackage{amsthm,amsmath,stmaryrd,bbm,hyperref,geometry,color,xcolor}
\usepackage{amssymb}
\usepackage[english]{babel}
\usepackage{graphicx}
\usepackage{amsfonts,amssymb}
\usepackage{verbatim}
\usepackage{enumitem}
\usepackage[all]{xy}
\usepackage{soul}
\usepackage{authblk}
\usepackage[dvipsnames]{xcolor}

\newcommand{\po}{\left(}
\newcommand{\pf}{\right)}

\newcommand{\E}{\mathbb E}
\newcommand{\R}{\mathbb R}

\newcommand{\N}{\mathbb N}

\newcommand{\dd}{\mathrm{d}}
\newcommand{\ee}{\mathrm{e}}
\newcommand{\na}{\nabla}
\newcommand{\X}{\mathbf{X}}

\newcommand{\pa}[1]{\left( #1 \right)}
\newcommand{\br}[1]{\left[ #1 \right]}
\newcommand{\abs}[1]{\left \vert #1 \right \vert }
\newcommand{\norm}[1]{\left \Vert #1 \right \Vert }

\newcommand{\x}{\mathbf x}

\newcommand{\calN}{\mathcal N}
\newcommand{\one}{\mathbf{1}}

\newcommand{\ph}{\varphi}
\newcommand{\Ha}{ {\mathcal H }}

\newtheorem{AssD}{Assumption}

\newtheorem{AssU}{Assumption}

\newcommand{\add}[1]{{\color{black}#1}}

\newcommand{\lj}[1]{{\color{black}#1}}

\newtheorem{thm}{Theorem}[section]

\newtheorem{lem}[thm]{Lemma}
\newtheorem{defi}[thm]{Definition}
\newtheorem{prop}[thm]{Proposition}
\newtheorem{cor}[thm]{Corollary}
\newtheorem{rem}[thm]{Remark}
\newtheorem{exa}[thm]{Example}

\title{The particle approximation of quasi-stationary distributions: concentration bounds in the uniform case.}
\author[1]{Lucas Journel\thanks{lucas.journel@unine.ch}}
\author[2]{Mathias Rousset\thanks{mathias.rousset@inria.fr}}
\affil[1]{Universit\'e de Neuch\^atel, Suisse.}
\affil[2]{IRMAR and Inria, University of Rennes, France.}
\date{ }

\begin{document}

\maketitle

\begin{abstract}
We study mean-field particle approximations of normalized Feynman-Kac semi-groups, usually called Fleming-Viot or Feynman-Kac particle systems. Assuming various large time stability properties of the semi-group uniformly in the initial condition, we provide explicit time-uniform $L^p$ and exponential bounds (a new result) with the expected rate in terms of sample size. This work is based on a stochastic backward error analysis (similar to the classical concept of numerical analysis) of the measure-valued Markov particle estimator, an approach that simplifies methods previously used for time-uniform $L^p$ estimates.
\end{abstract}

\section{Introduction}
Let $(X_t)_{t \geqslant 0}$ denotes a time-homogenous Markov process in some state space $E$, and $V: E \to \R $ a bounded measurable 'potential' function. Define the Feynman-Kac semi-group by
\[
Q^V_{t}(f)(x) := \E\po f(X_t)\ee^{-\int_0^t V(X_s) \dd s} \mid X_0=x \pf,
\]
and denote the associated normalized Feynman-Kac semi-group by
\[
\eta_t(f) := \frac{\int_E Q^V_t(f) \dd \eta_0}{ \int_E Q^V_t(\one) \dd \eta_0} .
\]
\add{Since $V$ is bounded, up to adding a constant}, one can assume $V\geqslant 0$ and $\eta_t$ can then be interpreted for each $t\geqslant 0$ as the distribution of the process $X_t$ with initial distribution $\eta_0$ conditioned by the survival event $\tau > t$, where $\tau$ is a random (stopping) time defining a killing event with killing rate per unit time given by $V(X_t)$.

We are interested in mean-field particle approximations of the flow $t \mapsto \eta_t$ (and in particular of its large time limit), introduced in this context in~\cite{BurHolMar00}. Those mean-field particle systems are sometimes called Fleming-Viot particle systems, or Feynman-Kac particle systems; and the importance of treating them as mean-field, as is done here, is largely due to work of P.~Del~Moral (see e.g. \cite{DelMic00} or \cite{del2004feynman,del2011concentration,del2012concentration} for the discrete time case).

Feynman-Kac particle systems associated with the pair $(X,V)$ are Markov processes of the form
\[
t \mapsto \X_t =(X^1_t, \ldots , X^N_t) \in E^N
\]
where $N$ denotes the number of particles. When $V \geqslant 0$, the specific variant usually called the Fleming-Viot process can be defined informally as follows:
\begin{itemize}
\item each particle evolves independently according to the distribution of $X$,
\item each particle is killed between time $t$ and $t+\dd t$ with probability $V(X_t)\dd t$,
\item when a particle is killed, it is immediately replaced by a clone of a uniformly chosen particle in the full sample of size $N$ (including the killed particle).
\end{itemize}
Other closely related variants of mean-field approximations can be defined, see their definition in Section~\ref{s-sec:math-set} below.

The estimator of the normalized Feynman-Kac semi-group is then the empirical measure of the particle system, denoted by
\begin{equation}\label{eq:emp-meas-part}
\eta^N_t(dx) :=\frac{1}{N}\sum_{j=1}^{N}\delta_{X^{j}(t)}(dx).
\end{equation}

The goal of this work is to prove, uniformly with respect to time, bias, $L^p$ and exponential concentration bounds \add{(sharp with respect to $N$)} on the difference $ \eta^N_t(f) - \eta_t(f)$, for any $f$ in an appropriate separating class of test functions. Such results are sometimes referred to as 'one-body' or 'weak' propagation of chaos. Results are obtained under a set of assumptions (see Assumption~\ref{assu:general} for details) that we will call here \emph{the uniform case}. The name comes from the various estimates on the large time convergence of the Feynman-Kac semi-group, which are required to hold \emph{uniformly in the initial condition}. One of the main feature of this set of assumptions for diffusions is the following. \add{Typically, first, we assume that there exists a unique Quasi-Stationary Distribution (QSD) $\eta_\infty$, and a unique bounded positive eigen-function $h: E \to \mathbb R_+$ with eigenvalue $\lambda \in \mathbb R$} defined by the eigen-problems
\[
\eta_\infty Q^V_t = \ee^{-\lambda t} \eta_\infty, \quad  Q^V_t(h)=\ee^{-\lambda t} h.
\]
Second, Assumption~\ref{assu:general} also requires the boundedness of $\log h$ and its gradient on the whole of $E$, which is quite restrictive. \add{Examples of diffusions satifying Assumption~\ref{assu:general} includes uniformly elliptic diffusions on a compact manifold with $V$ simply bounded. In the case of Euclidean diffusions with additive noise (see Section~\ref{sec:exemple} for examples of application), the uniform case of Assumption~\ref{assu:general} will hold true \textit{e.g.} for sufficiently confining drifts at infinity (at least super-linearly) and any killing potentials that can be written in the form $V = h^{-1} L(h) + cte$ with $\log h$ and its gradient bounded.}

Generalizations of the present \add{sharp estimates (especially the concentration estimates}) without this type of restriction remains an open problem, to the best of our knowledge.

More precisely, we obtain \add{as a preliminary} (Theorem~\ref{thm:var-lower-bound-h}) for i.i.d. initial conditions and each $N \geqslant 1$ 
\begin{equation*}
    \sup_{t\geqslant 0} \abs{\mathbb E \pa{\eta^N_t(f) - \eta_t(f)} } \leqslant \frac{C}{N} \mathcal N \pa{f},
\end{equation*}
as well as for $p\geqslant 2$
\begin{equation*}
    \sup_{t\geqslant 0} \pa{ \mathbb E \abs{\eta^N_t(f) - \eta_t(f)}^p }^{1/p}  \leqslant \frac{C}{\sqrt{N}} \mathcal N \pa{f},
\end{equation*}
in which $C$ is a constant and $\mathcal N$ a norm on test functions used in the uniform estimates of Assumption~\ref{assu:general}. The case of diffusions treated in this work requires \[\mathcal N \pa{f} = \norm{f}_\infty + \norm{ \nabla f }_\infty. \]

Theorem~\ref{thm:var-lower-bound-h} is not new as it was already proven (with few technical variants in the assumption) in \cite{Rousset06} (see aslo \cite{angeli2021limit,cloez2022uniform} which are based on the same ideas). However, the present work considerably simplify the technical tricks of the proof in~\cite{Rousset06}. Indeed, in the latter, a \emph{linearized} backward error analysis of the martingale 
\[
t \mapsto \int_E Q^{V- \lambda}_{T-t}(f) \dd \eta^N_t 
\]
is performed (for $t \leqslant T$, and $\lambda$ the principal eigenvalue of the problem). This approach is simpler at first, but it leads to spurious terms that are delicate to analyze uniformly in time. A more straightforward and conceptually cleaner approach, which is the one adopted here, consists in considering the exact backward error process defined by 
\[
t \mapsto \ph_t(\eta_{t}^N) := \frac{\int_E Q^{V}_{T-t}(f) \dd \eta^N_t}{\int_E Q^{V}_{T-t}(\one) \dd \eta^N_t} .
\]
The above backward error process (a terminology borrowed from numerical analysis) is obtained by considering the evolution of the particle system on $[0,t]$ followed by the exact $N=+\infty$ mean-field probability-valued evolution equation on $[t,T]$. Stochastic analysis of the variations when $t$ varies from $0$ to $T$ is then performed. Similar approaches, sometimes very involved, have already been carried out for other mean-field particle systems (see \textit{e.g.} \cite{mischler2013kac} which uses among many other tools this type of backward error decomposition to achieve parts of Kac's program in kinetic theory. \add{Note that in the present context a simple notion of differentiability of the non-linear semi-group $\eta \to \eta Q^V_t(f)/\eta Q^V_t(\one) $ is available thanks to its explicit nature, which greatly simplifies the approach.} See also \cite{de2021backward} for McKean diffusions in finite time).

Using exponentiation of the backward error process and a similar analysis, and under the same assumptions, we obtain 'exponential versions' of one-body propagation of chaos. This leads to exponential concentration bound, which reads for i.i.d. initial conditions, any $N \geqslant 1$ and any $u \in \R_+$:
\begin{equation*}
    \sup_{t\geqslant 0} \mathbb P \po  \abs{\eta^N_t(f) - \eta_t(f) }  \geqslant u \pf  \leqslant c_0 \exp \pa {-c  N \frac{u^2}{1+u}} .
\end{equation*}
for test functions satisfying $\mathcal N (f) \leqslant 1$. In the above $c_0 \leqslant 5$ is explicit and the small constant $c$ only depends on the constants of Assumption~\ref{assu:general}. This concentration estimate is the main new result of the present paper.

For older references studying the large time behavior of this type of particle systems in continuous time \add{with sharp estimates in $N$}, beyond $L^p$ estimates as in \cite{Rousset06,angeli2021limit,cloez2022uniform}, one can see~\cite{journel2022convergence} \add{in which sharp estimates in Wasserstein distance are obtained by an original coupling method, albeit with a quantitative restrictive assumption on the noise intensity. The concentration estimate in continuous time of this work is however new.} In discrete time, this program of using backward error semi-group analysis to study propagation of chaos for Feynman-Kac mean-field particle models has been exhaustively developed by Pierre Del~Moral and his co-authors (see \textit{e.g.} \cite{del2004feynman,del2011concentration,del2012concentration}). The most relevant result with respect to the present work is arguably~\cite{del2011concentration} in which time uniform and asymptotically precise concentration bounds are obtained for the particle approximation of quite general mean-field models. However, in this reference, the discrete nature of time together with Dobrushin-like mixing of the underlying Markov chain are essential in the error analysis and one cannot pass to the continuous time limit; as an example, the constants defined in Equation~$(A5)$ that drive the concentration inequality does not seem to scale well with the number of discrete time iterations $m$ required in order to achieve mixing.

Without the assumptions of uniformity considered here in Assumption~\ref{assu:general}, \add{sharp} propagation of chaos concentration estimates remain a fully open question. 

\add{Note that if one relaxes the sharpness with respect to the rate in $N$, there is a large body of work (based on different tools) establishing convergence of the stationary particle empirical distribution towards the minimal QSD $\eta_\infty$ under much weaker conditions, even when the QSD is not unique (the selection problem). When the QSD is unique, results have been obtained under fairly general setting, see \textit{e.g.} \cite{burdzy2000fleming, bieniek2012non} for the hard killing case where additional difficulties emerge to show non explosion of the particle system, \cite{Ferrari} for discrete state spaces and \cite{tough2022fleming} where an additional McKean-Vlasov interaction is considered. When the QSD is non-unique, in a more recent series of work (see \cite{villemonais2014general, asselah2016fleming, champagnat2021minimal,tough2025selection}), a proof is obtained for some various but specific cases with unbounded domain.} 

The paper is organized as follows. The precise mathematical description of the assumptions, the results and the class of examples satisfying the assumptions are given in Section~\ref{sec:results}. The main martingales associated with the backward error process are described in Section~\ref{sec:weak-back-error-anal}, especially in Proposition~\ref{prop:mart-var} and~\ref{prop:mart-exp} which contains the main general estimates. The proof of the theorems based on those estimates are studied in Section~\ref{sec:proof}. Finally, we prove Proposition~\ref{prop:mart-var} and~\ref{prop:mart-exp} in Section~\ref{sec:weak-back-error-anal} using generator calculus applied to the Feynman-Kac mean-field formalism.

\section{Precise Results and Examples}\label{sec:results}

\subsection{Notations}

$\mathcal P$ will denote the set of probability measures on $E$. We will denote integration of a bounded test function $f$ by $\eta$ with $\eta(f) = \int f \dd \eta$. Whenever an operator is linear, we will often (though not always) drop the parenthesis and denote \textit{e.g.} $\eta f$, $L f$ and so on. $\one$ will always denote the constant function equal to $1$. For any set $S$, $\mathcal{B}(S)$ denotes the set of measurable and bounded function from $S$ to $\R$. Bold letters will always denote a vector of the form
\[
\x = (x_1,\dots,x_N) \in E^N,
\]
where $N$ is the number of particles. $\mathcal P_N \subset \cal P$ will then denote the set of empirical measures of $N$ particles
\[
\mathcal P_N : = \left\{ \frac{1}{N}\sum_{j=1}^{N}\delta_{x_{j}}(dx); (x_1,\dots,x_N) = \x \in E^N \right\}.
\]

\subsection{Basic considerations}

In this work, we consider a standard probability space endowed with a right-continuous (up to null sets) filtration and assume one can construct from any stopping time $\tau$ and any initial condition $X_0=x$ copies of a time-homogenous Markov processes $(X_t)_{t \geqslant \tau}$ that are at least jointly progressively measurable with respect to randomness, time and initial condition ($(\omega,s,x) \mapsto X_s(\omega,x)$ is jointly measurable for $s \leqslant t$ when restricting on events defined by the filtration at time $t$ ). State spaces (as $E$) are endowed with a usual standard Borel $\sigma$-field derived from a Polish topology. In most cases, $X_{t}$ can be realized as a \textit{c\`adl\`ag} time-homogenous process, Markovian for its own natural filtration, and measurable with respect to initial conditions $X_0=x$. We however stress that the specific topology plays no role in the analysis carried out here.

Associated to the probability transition semi-group of this Markov process, and when this makes sense, we define the associated generator $L$ as its time derivative. More precisely, we will say that a bounded measurable function $f \in \mathcal B (E)$ belongs to the domain $\mathcal D(L)$ (a vector space containing constants) of the generator $L$ if, denoting $\pa{P_t (= Q^0_t)}_{t \geq 0}$ the semi group of the probability transitions, the following holds for any initial condition $x \in E$ and any time $t \geqslant 0$:
\[
\E\po \left|Lf(X_t)\right| \mid X_0 = x\pf < \infty,\text{ and }\partial_t P_t(f)(x) = LP_t(f)(x) = P_t(Lf)(x).
\]
In this case, we recall that
\[
t \mapsto f(X_t) - \int_0^t Lf (X_s) \dd s
\]
is a martingale (see~\cite{ethier2009markov}). \add{By extension, a time dependent function is said to belong to the associated domain if the above holds true for the extension $t \mapsto (X_t,t)$ with generator $L+\partial_t$}.

All constructed martingales are implicitly modified into their \textit{c\`adl\`ag} version (which is possible because the filtration is right-continuous). Note that with practically no loss of generality it is possible to restrict to \textit{c\`adl\`ag} processes for a well-chosen Polish topology of $E$, and to domains consisting only of  continuous bounded functions: $\mathcal D(L) \subset C_b(E)$. This will render all the processes constructed here automatically \textit{c\`adl\`ag}.

\subsection{Mathematical setting}\label{s-sec:math-set}

If $V$ is a bounded measurable function, let us denote by $\Phi$ the non-linear propagator defined for $t\geqslant 0$, $\eta \in \mathcal P$, $f\in \mathcal B(E)$ by
\begin{equation}\label{eq:FK_flow}
  \Phi_{t}(\eta)(f) := \frac{\eta Q_{t}^V(f)}{\eta Q_{t}^V(\one)}.
\end{equation}
This flow is equal to the normalized Feynman-Kac semi-group $t\mapsto\eta_t$. In all the proof, we will use the notation $\Phi$ in order to emphasise on the initial condition.
Fix any bounded jump kernel $K_\eta$ such that for all $f\in\mathcal B(E)$
\[
\eta(S_\eta(f))=  \eta(f)\eta(V) - \eta(Vf)
\]
where
\[
S_\eta(f)(x) = \int _{E} \po f(x') - f(x) \pf K_{\eta}(x,dx')
\]
is a mean field jump generator. The boundedness assumption on the jump kernel has to be understood as: there exists $K_\ast>0$ such that
\begin{equation}\label{eq:bound-jump-kernel}
\sup_{x\in E, \eta \in \mathcal P}\int_E K_\eta(x,\dd x') \leqslant K_\ast.
\end{equation} 

Then, the non-linear propagator satisfies the following Kolmogorov equation:
\begin{equation}\label{eq:Kolmo}
\left\{\begin{array}{lll}
  \partial_t \Phi_{t}(\eta)(f) &=&  \Phi_{t}(\eta) \po    L_{\Phi_{t}(\eta)} (f)   \pf, \\
\Phi_{0}(\eta)(f) &=&\eta(f),
\end{array}\right.
\end{equation}
for all $f\in \mathcal  D(L)$, $\eta\in\mathcal P$, and where
\[
L_\eta = L + S_\eta.
\]


Let us now rigorously define the particle system that will be used as an estimator of this non-linear propagator~\eqref{eq:FK_flow}. Fix a number of particle $N\in\N$ and an initial condition $\eta_0^N$. Then:
\begin{enumerate}
    \item Let $(X^1,\cdots,X^N)$ be a sample of i.i.d. realisation of the process $X$ with initial condition $\eta_0^N$.
    \item Draw a Poisson process with intensity $N K_{\ast}$ that will be used as possible branching times.
    \item At each time of this Poisson process, generically denoted by $\tau$, a particle is killed with probability \[\sum_{i=1}^N \frac{K_{\eta^N_\tau}(\one)(X^i_\tau)}{N K_{\ast}},\] 
    the killed particle being chosen to be $m$ with probability \[ \frac{K_{\eta^N_\tau}(\one)(X^m_\tau)}{ \sum_{i=1}^N K_{\eta^N_\tau}(\one)(X^n_\tau)},\]
    where $\eta^N$ is the empirical measure of the particle system defined in~\eqref{eq:emp-meas-part}.
    \item A uniformly chosen particle in the full sample (including the killed one) is then duplicated and a new trajectory is created using the Markov distribution of the underlying process $X$.
\end{enumerate}

\begin{exa} The two main choices for $K_\eta$ are the following:
\begin{itemize}
\item If $V\geqslant 0$, a possibility of jump kernel is given by the Fleming-Viot process defined by:
\[
K_\eta(x,\dd x') := V(x) \eta(\dd x')
\]
in which case each particle evolves independently according to the underlying generator $L$ (that can be interpreted as independent 'mutations' of particles), and is killed with rate $V(X_t)$, at which time any particle of the full sample of size $N$ (including the killed particle) is chosen uniformly at random, and is then duplicated, keeping the population size constant. Equivalently: each particle is killed with rate $\frac{N-1}{N}V(X_t)$, at which time a surviving particle is chosen uniformly at random, and is then duplicated, keeping the population size constant.

\item A lower variance choice is given by :
\[
K_\eta(x,\dd x') := \pa{V(x)-\eta(V) }_+ \eta(\dd x') + 
 \pa{V(x')-\eta(V) }_- \eta(\dd x')
\]
which can be interpreted as an independent combination of selective killing / uniform cloning with the positive part of the potential function $\pa{V(x)-\eta(V) }_+$, and uniform killing / selective cloning with the negative part of the potential function $\pa{V(x)-\eta(V) }_+$. For this choice $K_\eta \equiv 0$ if $V$ is any constant function.
\end{itemize}
\end{exa}

By construction, if a function $ f \in \mathcal B \po E^N \pf$ is such that for all $i\in \llbracket 1,N \rrbracket$, $x_i \mapsto f(\x) \in \mathcal D(L)$, then $f$ belongs to the domain of the generator $\mathcal L$ of the particle system $\mathcal D(\mathcal L)$ and we will write:
\[
\mathcal L f(\x) = \sum_{i=1}^N L_{\eta_\x}\po x_i \mapsto f(\x) \pf,
\]
where $\eta_\x$ denotes the empirical measure 
\[
\eta_\x(dx)=\frac{1}{N}\sum_{j=1}^{N}\delta_{x_{j}}(dx).
\]
Because of the exchangeability of the particle system, this defines as well a Markov process on the set $\mathcal P_N$ of empirical measure of $N$ particles, \add{endowed with the topology of convergence in distribution}. With a slight abuse of notation we write the generator as
\begin{equation}\label{eq:Lcal_def}
  \mathcal L(\ph)(\eta) = \mathcal L \po \x \mapsto \varphi(\eta_\x) \pf,
\end{equation}
and its domain as
\[
\mathcal D(\mathcal L) := \left\{ \ph:\mathcal P_N\to \R,\, x_i\mapsto \varphi\po \eta_\x \pf\in \mathcal D(L),\, \forall i\in \llbracket1,N\rrbracket\right\},
\]
This operator is the generator of the measure-valued Markov process $\eta^N$ defined in~\eqref{eq:emp-meas-part}.


Before stating the general assumption under which we are going to work, let us define the \textit{carr\'e du champ} operator:

\begin{defi}\label{def:carre-du-champs}
For any Markov generator $G$, we define the associated \textit{carr\'e du champ} operator by 
\[
\Gamma_G(f) :=  \frac{1}{2}\po G(f^2)-2f G(f)\pf,
\]
defined at least for all $f\in \mathcal D(G)$ such that $f^2\in \mathcal D(G)$. 
\end{defi}
The expression of $\Gamma_G$ depends on the exact expression of $G$. For instance, if $G$ is a diffusion operator in $\R^d$ (as can be the case for $G=L$) of the form
\[
G f  = \Delta f + b\cdot\na f,
\]
then the \textit{carr\'e du champ} is given by the square of the gradient
\[
\Gamma_G = |\na\cdot|^2.
\]
If $G$ is a bounded jump operator with kernel $K$
\[
G f(x) = \int_E \po f(x') - f(x)\pf K(x,\dd x'),
\]
then its \textit{carr\'e du champ} is given by:
\begin{equation*}
  \label{eq:gamma_jump}
  \Gamma_{G}(f)(x)=\frac{1}{2}\int_{E}(f(x')-f(x))^2 K(x,\dd x').
\end{equation*}
In particular, we get for the selection part $S_{\eta}$ of the dynamic
\begin{equation*}
  \Gamma_{S_\eta}(f)(x)= \frac{1}{2}\int_{E}(f(x')-f(x))^2 K_{\eta}(x,\dd x').
\end{equation*}
Moreover, the \textit{carr\'e du champ} can be shown to be positive, and is the restriction of a bilinear form, also called the \textit{carr\'e du champ}, defined by \[\Gamma_G(f,g) = G(fg) - fG(g) - gG(f),\]
so that it satisfies the Minkowski inequality
\begin{equation}\label{eq:minskovki}
\Gamma_G(f+g) \leqslant 2\Gamma_G(f) + 2\Gamma_G(g).
\end{equation}
The probabilistic interpretation of the \textit{carr\'e du champ}, proven in Lemma~\ref{lem:carre-du-champ-as-derivatives}, is given by the following infinitesimal variance formula
\begin{equation}\label{eq:Gamma_as_var}
2\Gamma_G(f)(x) = \lim_{t \to 0} \frac{1}{t} \mathbb V \mathrm{ar} \pa{ f(Y_t) \mid Y_0 =x },
\end{equation}
where $Y_t$ has generator $G$.

\subsection{Main assumptions and results}

In order to perform calculus with generators, the proof in the present work is carried out under the following implicit \textit{a priori} assumption: 

\begin{AssD}\label{assu:domain} 
There exists a vector subspace $\mathcal A \subset \mathcal B(E)$ of \add{measure-determining} bounded measurable functions containing constants such that for all $f\in \mathcal A$ and all $T \add{> t} \geqslant 0$, the functions:
\begin{itemize}
    \item $\add{(x,t)} \mapsto \pa{Q^V_{\add{T-t}}(f)}^k(x)$, $k=1,2$,
    \item $(\add{x,t}) \mapsto \ee^{N \frac{Q^V_{\add{T-t}}(f)(x) + a}{Q^V_{\add{T-t}}(\one)(x) + b}}$, $a \in \R$, $b > 0$, 
\end{itemize}
belong to the domain of $L\add{+\partial_t}$ with $\add{(x,t)\in E \times [0,T[}$. 
\end{AssD}

\add{
\begin{exa} Let $(X_t)_{t \geq 0}$ be a diffusion in $\R^d$ (or similarly in a compact differentiable manifold) with globally Lipschitz coefficients, and bounded diffusion matrix denoted $a:\R^d \to \R^{2d}$. $V$ is measurable and bounded. Assume that either : i) $V$ is $C^1(\R^d)$ with bounded first derivative, or ii) the diffusion coefficient $a$ of $(X_t)_{t \geq 0}$ is uniformly bounded from below (uniform ellipticity). Then Assumption~\ref{assu:domain} holds true for smooth compactly supported test functions \textit{e.g.} $\mathcal A = C^\infty_c(\R^d)$ (one could take $\mathcal A$ much larger but this choice is sufficient and convenient here).

Indeed, if $g: \R \to \R$ is any smooth real valued function and $V$ is bounded, it is easy to check that 
\[
(L+\partial_t) g\pa{Q^V_{\add{T-t}}(f)} = V \times g'\pa{Q^V_{\add{T-t}}(f)} + \frac12 g''\pa{Q^V_{\add{T-t}}(f)}\Gamma_L\pa{Q^V_{\add{T-t}}(f),Q^V_{\add{T-t}}(f)},
\]
so that if $V$, $f$, and the diffusion matrix $a$ (recall that $\Gamma_L = a( \nabla \, \lj{\cdot} , \nabla \, \lj{\cdot})$) are bounded, then the above can be bounded using estimates on the gradient of the Feynman-Kac semi-group $\abs{\nabla Q^V_{\add{T-t}}(f)}$. Upper bounds on the latter are widely studied in the literature. In particular, a uniform upper bound $\sup_{\R^d} \abs{\nabla Q^V_{\add{T-t}}(f)} < +\infty $ can be obtained for case i) by direct differentiation of the stochastic flow of a strong solution (a classical reference is~\cite{kunita1990stochastic}), and for case ii) using  Duhamel formula ($Q^V_{t} = Q^0_{t} - \int_0^t Q^0_{t-s}( V Q^{0}_s(f)) \dd s $) and standard diffusion semi-group gradient estimates ($\abs{ \nabla Q^0_{t}(f) } \leq \frac{c}{\sqrt{t}} \norm{f}_\infty$)   (one can see \textit{e.g.}~\cite{picard2002gradient}  and~\cite{PRIOLA2006244} for recent results with irregular coefficients).
\end{exa}
}
\begin{rem}\label{rem:domain}
 Without Assumption~\ref{assu:domain}, the proof is \textit{a priori} formal. However, the estimates obtained in Proposition~\ref{prop:mart-var} and~\ref{prop:mart-exp} (together with \textit{e.g.} the \textit{carr\'e-du-champ} bounds~\eqref{eq:born-carre-champ2}-\eqref{eq:born-carre-champ} of Assumption~\ref{assu:general}) still make sense without any assumption on the domains of the generators.
It is therefore in principle possible to obtain the present results without Assumption~\ref{assu:domain}, either by using stochastic analysis alone, or alternatively (with some technical work), by considering regularized versions of the initial process $X^\epsilon_t$ and then pass to the limit $\epsilon \to 0$ in the final estimates, the latter depending only on the bounds of Assumption~\ref{assu:general} below. In such generalized cases, upper bounds on the carr\'e-du-champs operator may be understood as upper bounds on the infinitesimal variance given by \eqref{eq:Gamma_as_var}. In most cases, however, it is not hard to check Assumption~\ref{assu:domain} by choosing an appropriate \add{measure-determining} space $\mathcal A$ of sufficiently regular test functions.
\end{rem}

The main quantitative assumption of our work is the following:
\begin{AssU}\label{assu:general} \add{Let $\mathcal A \subset \mathcal B(E)$ denote the same measure-determining vector subspace of bounded measurable functions on $E$ containing constants as in Assumption~\ref{assu:domain}. Let $\mathcal A$ be endowed with two semi-norms $\mathcal N_1, \mathcal N_2$}. The following holds:
\begin{itemize}
    \item $V$ and $K_\eta$ are bounded (in the sense of~\eqref{eq:bound-jump-kernel}).
    \item There exists a real $\lambda \in \R$ and constants $C_+,c_->0$ such that for all $t \geqslant 0$ and all $x \in E$
    \begin{equation}\label{eq:borne-semi-group}
    c_-  \leqslant  e^{\lambda t} Q^V_{t}(\one)(x) \leqslant C_+ .
    \end{equation}
    \item There exists a (\add{necessarily unique}) probability distribution $\eta_\infty$ (the so-called quasi-stationary distribution) and a non-increasing $w_1\in L^1(\R_+)\cap L^2(\R_+)$ such that for all 
    $f\in \mathcal A$, $ t\geqslant 0$, $x\in E$
    \begin{equation}\label{eq:conv-semi-group}
    \left|e^{\lambda t}Q^V_t(f- \eta_\infty(f))(x)  \right| \leqslant w_1(t)\mathcal N_1(f).
    \end{equation}
     \item There exist a constant $C$ such that for all $t\geqslant 0$
    \begin{equation}\label{eq:born-carre-champ2}
    \ee^{\lambda t} \sqrt{\Gamma_L(Q^V_{t}(\one))} \leqslant C .
    \end{equation}
    \item There exists some non-increasing $w_2\in L^1(\R_+)\cap L^2(\R_+)$ such that for all $f\in \mathcal A$ and $t \geqslant 0$
    \begin{equation}\label{eq:born-carre-champ}
     \ee^{\lambda t} \sqrt{\Gamma_L(Q^V_{t}(f - \eta_\infty(f) ))} \leqslant w_2(t)\mathcal N_2(f).
    \end{equation}
\end{itemize}
\end{AssU}
\add{
\begin{rem}  Under the Harnack-type condition~\eqref{eq:borne-semi-group} of Assumption~\ref{assu:general}, the convergence condition~\eqref{eq:conv-semi-group} is essentially equivalent to convergence in time of
\[
\frac{\nu Q^V_t}{\nu Q^V_t(\one)}
\]
towards the unique QSD $\eta_\infty$, uniformly in the initial distribution $\nu$ (see~\eqref{eq:conv-non-linear-flow} for details) with speed $w_1$ ($w_1$ is typically exponential). Exactly this behavior can be obtained using the general theory of~\cite{conv-QSD} under Lyapounov-type condition and local Doeblin minorization (or Dobrushin criterion) instead of Condition~\eqref{eq:conv-semi-group} of Assumption~\ref{assu:general}. We avoid this type of setting in the presentation of the conditions in Assumption~\ref{assu:general} for two reasons. First we stick to conditions on the semi-goup $Q^V_t$ that are of direct use in the proofs of this work. Second, the carré-du-champ (or gradient for diffusions) estimates~\eqref{eq:born-carre-champ2}-\eqref{eq:born-carre-champ} is not obtained in such general settings. In Section~\ref{sec:exemple}, which treats diffusion examples we rather resort to a $h$-transform technique (where $h$ is the necessarily unique bounded positive eigen-function) in order to import results on Markov semi-groups.

\end{rem}
}
\add{
\begin{exa} If $(X_t)_{t \geq 0}$ is a uniformly elliptic diffusion with Lipschitz coefficients on a compact Riemannian smooth manifold and $V$ is only bounded measurable, then Assumptions~\ref{assu:general}-\ref{assu:domain} are satisfied with $\mathcal A = C^\infty$, $\calN_1 = \norm{ \,\,\,}_\infty$ and $\calN_2 = \norm{ \nabla \,\,\,}_\infty$, see Lemma~\ref{lem:compact_conv}.

In Section~\ref{sec:exemple}, specific conditions and examples ensuring Assumptions~\ref{assu:general} (and~\ref{assu:domain}) for a large class of diffusions in $\R^d$ (in particular the quite specific boundedness of $\log h$ and its gradient, $h$ being the positive eigenvector of $L-V$), will be provided.
\end{exa}
}
Under this assumption as well as the previous one, we first prove a theorem on the $L^p$ error,
\begin{thm}\label{thm:var-lower-bound-h}
Under Assumptions~\ref{assu:domain} and~\ref{assu:general}, for all $p \geq 1$, there exists a constant $C_p>0$ (which depends explicitly on the constants in  Assumption~\ref{assu:general}) such that for all $f\in\mathcal A$, and all $N\in\N$
\[ \sup_{t\geqslant 0F} \E \po \po \eta^N_t(f) - \Phi_t(\eta_0^N)(f) \pf^p\pf \leqslant C_p\frac{ \mathcal N_1(f)^p + \mathcal N_2(f)^p + \|f\|_{\infty}^p}{N^{p/2}},\]
\[ \sup_{t\geqslant 0} \E \po \po \eta^N_t(f) - \Phi_t(\eta_0^N)(f) \pf^2\pf \leqslant C_2\frac{ \mathcal N_1(f)^2 + \mathcal N_2(f)^2}{N},\]
and
\[
\sup_{t\geqslant 0} \left | \E \po  \eta^N_t(f) - \Phi_t(\eta_0^N)(f)  \pf \right| \leqslant C_1\frac{ \mathcal N_1(f) + \mathcal N_2(f)}{N}.
\]
\end{thm}
\add{This theorem has already been obtained in~\cite{Rousset06} for reversible diffusions, using very involved technical tricks, but without resorting to the gradient conditions~\eqref{eq:born-carre-champ2}-\eqref{eq:born-carre-champ}. The present proof is much more straightforward.}

Next is a new theorem related to concentration inequalities. 

\begin{thm}\label{thm:exp-bound}
Under assumptions~\ref{assu:domain} and~\ref{assu:general}, there is a constant $C>0$ (which depends explicitly on the constants in Assumption~\ref{assu:general}), such that for all $f\in\mathcal A$, and all $N \in\N$:
\[ 
\sup_{t\geqslant 0} \E \po e^{N\po\eta^N_t(f) - \Phi_t(\eta_0^N)(f)\pf} \pf  \leqslant e^{Ce^{C\|f\|_{\infty}} \po  N\po \mathcal N_1(f)^2 + \mathcal N_2(f)^2 \pf +  \po \mathcal N_1(f) + \mathcal N_2(f)\pf\pf }.
\]
\end{thm}
As mentioned in the introduction, the interest of Theorem~\ref{thm:exp-bound} resides in the following concentration bounds that follows from it:
\begin{cor}\label{cor:concentration}
Under the assumptions of Theorem~\ref{thm:exp-bound}, there exists a (small) $c>0$ that only depends on the constants invovled in Assumption~\ref{assu:general} such that for all $f\in\mathcal A$ verifying
\[
\norm{f }_\infty + \mathcal N_1(f) + \mathcal N_2(f) \leqslant 1,
\]
all initial condition of the particle system $\eta^N_0$ and all $u \geqslant 0$ it holds:
\[
\sup_{t \geqslant 0} \mathbb P\po \left | \eta^N_t(f) - \Phi_t(\eta_0^N)(f)\right| \geqslant u \pf \leqslant 2\ee^{\frac12 } \exp \pa{- N c \, \frac{u^2}{1+u}},
\]
and in particular
\[
\limsup_{t \to + \infty}\mathbb P\po \left | \eta^N_t(f) - \eta_\infty(f) \right| \geqslant u \pf \leqslant 2\ee^{\frac12 } \exp \pa{- N {c}{} \, \frac{u^2}{1+u}}
\]
where $\eta_\infty$ is the quasi-stationary distribution.
\end{cor}

\begin{rem} The left hand sides in Theorem~\ref{thm:var-lower-bound-h}, Theorem~\ref{thm:exp-bound} and Corollary~\ref{cor:concentration} are invariant by addition of a constant to the test function $f \in \mathcal A$. As a consequence, the three norms (or semi-norms) $ \mathcal N_1(f) $, $ \mathcal N_2(f)$ and $ \|f\|_{\infty} $ can be replaced by $ \inf_c \mathcal N_1(f -c)$, $ \inf_c \mathcal N_2(f- c) $ and  $ \inf_c \norm{f-c}_\infty$ respectively. We have avoided to use them in the statements in order to lighten notation. 
\end{rem}
\begin{proof}[Proof of Corollary~\ref{cor:concentration}] Denote $\overline f= f - \Phi_t(\eta_0^N)(f)$. For any $\alpha >0$, write using Theorem~\ref{thm:exp-bound}:
\begin{align*}
\mathbb P\po \eta^N_t\po \overline f \pf \geqslant u \pf 
&= \mathbb P\po \exp \po \alpha N\po \eta^N_t\po \overline f \pf \pf \pf \geqslant \exp\po \alpha Nu\pf \pf  
\\ &\leqslant e^{-\alpha Nu}\E \po e^{N\po\eta^N_t(\alpha \overline f)  \pf} \pf  
\leqslant \ee^{-\alpha Nu}\ee^{C e^{C\alpha} \po  \alpha^2 N +  \alpha \pf}.
\end{align*}
Now fix $\kappa$ such that $C \kappa e^{C \kappa} \leqslant  \frac12$ and take
$
\alpha := \kappa \frac{u}{1+u}.
$
We have that $\alpha \leqslant \kappa$,
as well as:
\begin{align*}
\mathbb P\po \eta^N_t\po \overline f \pf \geqslant u \pf & \leqslant \exp{\pa{ - N \kappa u^2/(1+ u) + \frac12 N \kappa u^2/(1+ u)^2 + u/(2+2 u)}} \\
& \leqslant  \exp{ \pa{ - N \kappa \frac12 u^2/(1+ u) +  1/2} } .
\end{align*}
The deviations from the left are obtained in the same way, taking $-f$ instead of $f$, and this concludes the proof.
\end{proof}

\begin{rem} 
The concentration inequality in Corollary~\ref{cor:concentration} can be compared to Chernov-based inequalities for i.i.d. variables. For instance, Bernstein inequality yields that if $(X_n)_{1\leqslant n \leqslant N}$ are centered with unit variance and $\abs{X_n} \leqslant M $ almost surely, then
\[
\mathbb P \br{ \frac1N \sum_{n=1}^N X_n \geqslant u } \leqslant \exp\pa{ - N\frac{u^2}{2}\frac{1}{1+M u /3}  },
\]
which is also a correction to Gaussian concentration (Gaussian concentration is of order $O(u^2)$) with an affine correction term $1+u$. In particular, such standard concentration inequalities in the i.i.d. case tell us that the leading order large deviations term $-N u^2$ in the exponential (for $N$ large and $u$ small) is optimal, up to (here sub-optimal) constant $c$.
\end{rem}

\subsection{Examples}\label{sec:exemple}

We provide in this section a class of examples in $\R^d$ satisfying Assumption~\ref{assu:general}. We consider diffusions of the form 
\begin{equation*}
\dd X_t = b(X_t)\dd t + \sqrt{2}\dd B_t,
\end{equation*}
with $b$ a locally Lipshitz vector field. If we suppose the existence of a positive eigenfunction $h$ satisfying $Lh - Vh = -\lambda h$ (\add{that is necessarily unique under Condition~\eqref{eq:borne-semi-group})}, we may then define $P^h$ as the $h-$transform of $Q^V$ by
\[
P^h_t(f) = e^{\lambda t}h^{-1}Q^V_t(hf),
\]
for all $f$ such that $hf\in\mathcal B(E)$.
Using this semi-group, one \add{can then expect} to check Assumption~\ref{assu:general} using known results on the theory of large time behavior of diffusion processes. An exhaustive analysis of results related to the convergence condition~\eqref{eq:conv-semi-group} can be found in~\cite{conv-QSD}. Additionally,~\cite{conv-QSD-uniform} focus on a Lyapounov criteria ensuring uniform (with respect to the initial condition) convergence. 
For simplicity, we rather give a straightforward proof using the classical reference~\cite{down1995exponential}. 

\begin{lem}\label{lem:conv_ex} Let $V$ be bounded and assume that there exists a \add{(necessarily unique)} positive eigenfunction $Lh - Vh = \lambda h$ such that
\begin{itemize}
    \item $\log h$ is bounded,
    \item There is a $\delta > 0$ and a \emph{bounded} Lyapounov function $\phi \geqslant 1$ (in the domain of $L$) such that, outside a compact set:
    \begin{equation}\label{eq:Lyap}
    L(\phi) \leqslant - (\delta - V + \lambda) \phi .
    \end{equation}
\end{itemize}

Then all of Assumption~\ref{assu:general} except the carr\'e-du-champ bounds \eqref{eq:born-carre-champ2}-\eqref{eq:born-carre-champ} hold true; in particular~\eqref{eq:conv-semi-group} is satisfied with $w_1(t)$ a decreasing exponential and $\calN_1$ the supremum norm.
\end{lem}
\begin{proof}
First remark that since $h$ is an eigenfunction with bounded logarithm,
\[
e^{-\lambda t}\inf_E h  \leqslant e^{-\lambda t}h = Q^{V}_t(h) \leqslant \sup_{E} h Q^{V}_t(\one), 
\]
so that $e^{\lambda t}Q^{V}_t(\one)\geqslant c_-$. Similarly
\[
Q^{V}_t(\one) \leqslant (\inf_E h)^{-1}Q^{V}_t(h) \leqslant e^{-\lambda t}(\inf_E h)^{-1} \sup_E h,
\]
and hence~\eqref{eq:borne-semi-group} holds.

For the initial process $X$, the general property of (Lebesgue) irreducibility used in the classical literature follows from the existence of a Lebesgue density $p_t(x,y)$ for the probability transitions of the process; the fact that compact sets are petite (in the terminology of \cite{down1995exponential}) follows from the fact that at least on some time interval this density is bounded away from $0$ ($p_t(x,y) \geqslant \delta > 0$ for all $x,y$ in a compact set). The $h$-transform also admits a density $p^h$ with respect to the Lebesgue measure satisfying $p^h_t(x,y) \in \left [ h^{-1}(x)p_t(x,y)h(y)\ee^{t \inf V} , h^{-1}(x)p_t(x,y)h(y)\ee^{t \sup V} \right ]$. Now, because $\log h$ is bounded, $P^h_t$ inherits the property that compact sets are petite.

We can now apply condition $(\tilde{\mathcal D})$ of~\cite{down1995exponential} to $P^h$ with Lyapounov function $\phi h^{-1}$, which by construction of the $h$-transform belongs the domain of the generator of $P^h$, and thus satisfies:
\[
L^h(\phi h^{-1}) = h^{-1} L \phi - h^{-1} L(h) \leqslant - \delta h^{-1} \phi .
\]
Then, \cite[Theorem 5.2]{down1995exponential} yields that $P^h$ is exponentially uniformly exponentially ergodic, and since again $\log h$ is bounded we exactly get~\eqref{eq:conv-semi-group} with $w_1(t)$ a decreasing exponential and $\calN_1$ the supremum norm.
\end{proof}

\begin{lem}\label{lem:conv_ex_2}
Assume in addition to the assumptions in Lemma \ref{lem:conv_ex} that:
\begin{itemize}
    \item $\nabla \log h$ is bounded,
    \item for all $x,y \in \R^d \setminus K$ outside a compact set $K$,
\[
(x-y)\cdot \pa{b(x) - b(y)} \leqslant -\kappa |x-y|^2.
\]
\end{itemize}
Then, the carr\'e-du-champ bounds~\eqref{eq:born-carre-champ2}-\eqref{eq:born-carre-champ} holds, and hence the whole of Assumption~\ref{assu:general} is true.
\end{lem}

\begin{proof} 
The proof is based on an application of the results in~\cite{couplage-non-convexe} (based on coupling) to the $h$-transform $P^h$ in order to show that:
\begin{equation}\label{eq:W_1}
\mathcal W_1 ( \mu P^h_t , \nu P^h _t   ) \leqslant C\ee^{- c t} \mathcal W_1 ( \mu , \nu )
\end{equation}
for some constants $c > 0$ and $C$, and where $\mathcal W_1$ is the Euclidean Wasserstein distance. The Kantorovitch dual formulation of $\mathcal W_1$ enables to obtain convergence estimates on the gradient of diffusion semi-groups.

We now give more details. $P^h$ is a semi-group with generator
\[
L = \Delta + b\cdot \na + \na \ln(h)\cdot \na = \Delta + \tilde b\cdot \na.
\]
Here $\tilde b$ still satisfies the assumptions of~\cite{couplage-non-convexe} if $\nabla \log h$ is bounded, in particular we can apply Corollary $2$ (line $2$) and obtain the exponential convergence in Wasserstein distance~\eqref{eq:W_1}. \add{Using the Kantorovitch dual formulation of $\mathcal W_1$ as a Lipschitz operator norm $\mathcal W _1 (m) = \sup_{\phi} \frac{m(\phi)}{\norm{\phi}_{Lip}}$ (a general argument that holds on any Riemann manifold is~\cite[Theorem 2.2 (p=1)]{gradient-estimate})}, one gets that there exists $C,c>0$ such that for all $t\geqslant 0$ and $f$ with bounded gradient:
\begin{equation}\label{eq:grad_Ph}
\|\na P^h_t(f)\|_{\infty} \leqslant Ce^{-ct}\|\na f\|_{\infty}.
\end{equation}
Now, we have:
\[
Q^V_t(f) = e^{-\lambda t}h P^h(h^{-1}f),
\]
so that
\[
\na Q^V_t(f) = e^{-\lambda t} h \na P^h(h^{-1}f) + \frac{\na h}{h} Q^V_t(f).
\]
Since $\log h$ and $\nabla \log h$ are assumed bounded, we obtain using Lemma~\ref{lem:conv_ex} the \textit{carr\'e-du-champs} estimates \eqref{eq:born-carre-champ2}-\eqref{eq:born-carre-champ}.
\end{proof}

\add{Note that checking the uniform bound on $\log h$ and its gradient for a given pair $(L,V)$ may be uneasy in practice. By first prescribing $h$ and from it deducing $V$, it is however possible to generate a large class of diffusions in a non-compact setting like $\R^d$ satisfying our assumptions. We provide a typical example below.}

\begin{exa}
Let $d=1$, $b(x)  = - x^2$. Fix any $h:\R\to \R_+^*$ such that $\log h $ and its two first derivatives are bounded, and write $V = h^{-1} L h$. Then, $\lambda = 0$, and if $m=\max(-\inf_\R V,0)$, then
\[
\phi(x) = 2\po 1 - \frac{1 + m}{|x|}\pf \implies L\phi(x) \leqslant -2\po 1 + m\pf \leqslant -\phi \po 1 -  V \pf
\]
for $|x|$ great enough, and the assumptions from Lemma~\ref{lem:conv_ex} and~\ref{lem:conv_ex_2}, and thus Assumptions~\ref{assu:domain}-\ref{assu:general}, are satisfied.
\end{exa}

\begin{rem}[\add{Counter-example}] The existence of \emph{bounded} Lyapounov functions satisfying ~\eqref{eq:Lyap} for Euclidean diffusion with additive noise typically require a confining drift that grows strictly more than linearly at infinity, as in the above example. For instance if $V = 0$ and $\lambda =0$, and $L$ generates an Orstein-Uhlenbeck process (in one dimension $b(x) = - \mu x$ with $\mu > 0$), the large time convergence to equilibrium is not uniform in the initial condition. 
\end{rem}

\add{Our assumptions are also generically satisfied in the case of a diffusion on a compact manifold.
\begin{lem}\label{lem:compact_conv} Let $(X_t)_{t \geq 0}$ be an uniformly elliptic diffusion with Lipschitz coefficients on a compact differentiable manifold and $V$ be measurable and bounded, then Assumptions~\ref{assu:domain}-\ref{assu:general} are satisfied.
\end{lem}
\begin{proof} We give a sketch of a proof that follows the arguments in $\R^d$ above. In a compact setting, by standard Krein-Rutman arguments, there exists a principal eigenfunction $h$ which by elliptic regularization is positive and smooth (see \textit{e.g.} \cite{pinsky1995positive}). By compactness, $\log h$ and $\nabla \log h$ are bounded. 

Let us recall a very classical argument showing the Wasserstein exponential convergence \ref{eq:W_1} in the compact setting. By uniform ellipticity and compactness, the transition kernel admits a strictly positive smooth density. Hence a Doeblin minorization condition holds, and by standard results (e.g. Meyn and Tweedie), the semigroup is exponentially ergodic in total variation, in particular for all $0 \leq \delta t \leq t$:
\[
\norm{ \mu P^h_t , \nu P^h _t }_{tv} \leqslant C\ee^{- c (t-\delta t)} \norm{\mu P^h_{\delta t} , \nu P^h_{\delta t}  }_{tv}.
\]
For a bounded distance, the Riemannian $W_1$ norm is bounded by the total variation, and by elliptic regularization, on a compact manifold, $\norm{\mu P^h_{\delta t} ,\nu P^h_{\delta t} }_{tv} \leq c_{\delta t} \mathcal W _ 1 (\mu , \nu) $. This yields \ref{eq:W_1}. 

The rest follows exactly the proofs of Lemma~\ref{lem:conv_ex}-\ref{lem:conv_ex_2} which still hold on a general Riemannian space.
\end{proof}
}

\section{Proofs of the theorems}\label{sec:proof}

\subsection{Decompositions of the backward error process}

Fix $f\in \mathcal A$, $T>0$. We consider functions on $[0,T] \times \mathcal P$ defined by 
\begin{equation*}
\ph_t(\eta) := \Phi_{T-t}(\eta)(f) - \Phi_{T}(\eta_0^N)(f)
\end{equation*}
where $(\Phi_t)_{t \geqslant 0}$ is the mean-field measure-valued evolution semi-group~\eqref{eq:FK_flow}. Recall that $t \mapsto \eta^N_t$ denotes the empirical distribution of the particle system. The present work is based on the stochastic analysis of the backward error process $t \mapsto \ph_t(\eta^N_t)$ obtained by propagating the particle approximation over the time interval $[0,t]$, and then the limit evolution semi-group on the time interval $[t,T]$.

Recall that if $(x,t) \mapsto \psi_t(x)$ is a time-dependent function belonging to the domain of the generator $\partial_t + L$ of any time-homogenous Markov process $X$, then the process
\[
M_t = \psi_t(X_t) - \int_0^t (\partial_s + L) \psi_s ( X_s) \dd s
\]
is a martingale. If moreover $\psi_t^2$ is also in the domain, then the predictable (or conditional) quadratic variation denoted by $\langle M \rangle_t$ of the martingale, defined as the predictable compensator of the usual quadratic variation $\br{M}_t$, is given by
\begin{equation} \label{eq:pred_mart}
\langle M \rangle_t = \int_0^t 2 \Gamma_L(\psi_s)(X_s) \dd t .
\end{equation}
By definition $\langle M \rangle_t$ is the only predictable (in our case even differentiable with respect to time) increasing process such that $t \mapsto \br{M}_t- \langle M \rangle_t$ is a (pure jump) martingale. We refer to \cite[Chapter~3]{protter2005stochastic} as a classical reference for advanced stochastic calculus (possibly with jumps) covering semi-martingales, predictable projections,  \textit{et cetera}. In the present work we will only use the fact that using \eqref{eq:pred_mart}, the process $t \mapsto \br{M}_t- \langle M \rangle_t$  has finite variation, hence is a pure jump martingale. In the same way, if $\ee^{\psi_t}$ belongs to the domain of the generator $\partial_t+L$, the following process
\[
t \mapsto \exp\pa{ \psi_t(X_t)  -  \int_0^t \ee^{-\psi_s(X_s)} (L+\partial_s)(\ee^{\psi_s}) (X_s) \dd s}
\]
is a (so-called 'exponential') martingale.

We will denote the deviation of a test function $f\in\mathcal B(E)$ with respect to the semi-group~\eqref{eq:FK_flow} by
\begin{equation}\label{eq:fun-centered}
\overline f : = f - \Phi_{T}(\eta_0^N)(f).
\end{equation}
Recalling that $\mathcal P_N $ is the set of empirical distribution of $N$ particles, we will denote for all $\eta = \eta_{\x} \in \mathcal P_N$:
\[
\eta^{(i)} := \frac{1}{N} \sum_{j \neq i} \delta_{x_i},
\]
the measure of mass $\frac{N-1}{N}$ obtained by removing particle $i$.

The proofs of Theorem~\ref{thm:var-lower-bound-h} and~\ref{thm:exp-bound} rely on the following two results.
\begin{prop}\label{prop:mart-var} 
Suppose Assumption~\ref{assu:domain}. Let $T > 0$, $f\in \mathcal A$ be given, $V$ be bounded. Then the error process $\ph_t(\eta^N_t)$ has a (Doob-Meyer) decomposition of the form
\begin{equation}\label{eq:mart-var-with-rest}
\ph_t \po \eta_t^N \pf = M_t^{1,T} + \frac{1}{N}\int_0^t R\br{\ph_s } \pa{\eta^N_s} \dd s,
\end{equation}
where $ t \mapsto M^{1,T}_t$ is a martingale (for the underlying filtration) and the rest term  (the time derivative of the predictable finite variation part) satisfies for all $N\in \N$, $\eta\in\mathcal P_N$ and $t \in [0,T]$: 
\begin{multline*}
\left|R\br{\ph_t} \po \eta \pf \right| \leqslant \max_{1\leqslant i \leqslant N}\frac{4}{\po\eta^{(i)} \po Q^{V}_{T-t}\po \one \pf  \pf\pf^2} \\ \times \po  \sqrt{\eta \Gamma_{L_\eta}\po Q^{V}_{T-t}\po \overline f \pf\pf}\sqrt{\eta \Gamma_{L_\eta}\po Q^{V}_{T-t}\po \one \pf\pf} + \Phi_{T-t}(\eta)\po \overline f \pf\eta\Gamma_{L_\eta}\po  Q^{V}_{T-t}\po \one \pf \pf \pf.
\end{multline*}
Moreover, the quadratic variation $\left< M^{1,T} \right>$ of $M^{1,T}$ satisfies for $0\leqslant t \leqslant T$ 
\begin{multline}\label{eq:borne-var-quad}
N\left< M^{1,T} \right>_t \\ \leqslant  \int_0^t \max_{1\leqslant i \leqslant N}\frac{4}{\po\eta^{N,(i)}_t \po Q^{V}_{T-t}\po \one \pf \pf\pf^2}\po  5 \eta_s^N \Gamma_{L_{\eta^N_s}}\po Q^{V}_{T-s}\po \overline f \pf\pf + 8 \Phi^2_{T-s}(\eta_s^N)\po \overline f \pf \, \eta_s^N\Gamma_{L_{\eta^N_s}}\po Q^{V}_{T-s}\po \one \pf \pf \pf \dd s.
\end{multline}
\end{prop}

\begin{rem} Some remarks.

\begin{itemize}
\item By construction of the martingale $M^{1,T}_t$, the rest term $R\br{\ph }(\eta)$ is equal to 
\[
R\br{\ph_t}(\eta) = N \pa{\partial_t + \mathcal L } \ph_t(\eta).
\]
In the proof of Proposition~\ref{prop:mart-var} we will use the fact that when $\mathcal L$ is a mean-field particle generator, this rest term (which depends on $N$) is of order $\mathcal O_N(1)$.

\item As will be computed in the proof, the predictable quadratic variation is in fact equal to
\[
\left< M^{1,T} \right>_t = \int_0^t 2\Gamma_{\mathcal L}(\ph_s)(\eta^N_s) \dd s = \frac1N \int_0^t R\br{\ph_s^2}(\eta_s^N) - 2 \ph_s(\eta_s^N) R\br{\ph_s}(\eta_s^N) \dd s .
\]
where the first identity is the usual formula for martingales constructed out of functions (in the generator domain) of Markov processes. Here, \[
R\br{\ph_t^2}(\eta) = N \pa{\partial_t + \mathcal L } \ph_t^2(\eta)
\]
is the integrand of the finite variation part from the Doob-Meyer decomposition of the square error process $\ph_t^2(\eta_t^N)$.
\item The content of Proposition~\ref{prop:mart-var} is thus to deliver upper bounds on these rest terms in the specific Feynman-Kac case where the non-linear flow $\Phi_{T-t}$ is explicit.
\item As explained in Remark~\ref{rem:domain}, the proof is carried out when $\ph$ and $\ph^2$ belong to the domain of the generator $\mathcal L$ of the particle system (which follows from Assumption~\ref{assu:domain}). As already explained in Remark~\ref{rem:domain}, the result certainly still holds in general. A rigorous general statement would require the use of super- and sub-martingales instead of martingales.
\end{itemize}
\end{rem}
A similar result holds in the exponential case. In order to state it, we define for $f\in\mathcal D\po L_\eta\pf$ such that ${\rm e}^{f}\in\mathcal D\po L_\eta\pf$, $\eta\in\mathcal P$, the exponential \textit{carr\'e du champ} by
\[
\Gamma_{\eta}^{\rm exp}(f) := {\rm e}^{-f}L_{\eta}( {\rm e}^{f})-L_{\eta}(f).
\]
We also need the notion of flat derivative of measure-valued test function $\ph_t(\eta)$ with respect to a probability measure $\eta$, that we will precisely define in Section~\ref{s-sec:derivative}. For now, we simply need that this derivative is a function $\frac{\delta \ph }{ \delta \eta}\in\mathcal B(\mathcal P\times E)$ satisfying usual chain rules. 
If $\Gamma$ is either a \textit{carr\'e du champ} or an exponential \textit{carr\'e du champ}, we use the notation:
\[
\eta\Gamma\po \frac{\delta \ph }{ \delta \eta} \pf = \int_E \Gamma\po \frac{\delta \ph }{ \delta \eta}(\eta,\cdot) \pf(x) \eta(\dd x).
\]

\begin{prop}\label{prop:mart-exp} 
Suppose Assumption~\ref{assu:domain}, as well as Condition~\eqref{eq:borne-semi-group}. Let $T > 0$, $f\in \mathcal A$ be given.
Then the exponential of $N$ times the error process has a decomposition of the form
\begin{equation}\label{eq:mart-exp-with-rest}
\exp\po N\ph_t\po\eta_t^N \pf \pf =  \exp\po N\int_0^t \left[ \eta_s^N \Gamma_{L_{\eta^N_s}}^{\rm exp}\left(\frac{\delta \ph_s}{ \delta \eta}\right) + \frac{1}{N} R^{\rm exp}\br{\ph_s} \po \eta_s^N \pf \right]\dd s \pf  \times M_t^{\rm exp,T}
\end{equation}
where $t \mapsto  M_t^{\rm exp,T}$ is a martingale. In the above, for all $\eta\in\mathcal P_N$, $0\leqslant t \leqslant T$, the exponential \textit{carr\'e du champ} satisfies:
\[
\eta\Gamma_{\eta}^{\rm exp}\left(\frac{\delta \ph_t}{ \delta \eta}\right) \leqslant \frac{e^{C\|f\|_{\infty}}}{\po\eta \po Q^V_{T-t}(\one)  \pf\pf^2}\po  \eta \Gamma_{L_\eta}\po Q^V_{T-t}\po \overline f \pf\pf + \Phi_{T-t}(\eta)\po \overline f \pf^2\eta\Gamma_{L_\eta}\po Q^V_{T-t}(\one) \pf \pf;
\]
and the rest term satisfies 
\begin{multline*}
\left| R^{\rm exp}\br{\ph_t}(\eta) \right| \\ \leqslant \max_{1\leqslant i \leqslant N} \frac{2\ee^{C\|f\|_{\infty}}}{\po\eta^{(i)} \po Q^V_{T-t}(\one) \pf\pf^2} \po  \sqrt{\eta \Gamma_{L_\eta}\po Q^V_{T-t}\po \overline f \pf\pf}\sqrt{\eta \Gamma_{L_\eta}\po Q^V_{T-t}(\one)\pf} + \Phi_{T-t}(\eta)\po \overline f \pf\Gamma_{L_\eta}\po Q^V_{T-t}(\one) \pf \pf,
\end{multline*}
for some $C$ that depends only on the constants $c_-,C_+$ from Condition~\eqref{eq:borne-semi-group}.
\end{prop}

\begin{rem} As above:
\begin{itemize}
    \item The martingale in Proposition~\ref{prop:mart-exp} is the usual exponential martingale associated with Markov processes and the sum of  
    the exponential carr\'e-du-champs term and of the rest term are given by (for any $(t,\eta) \mapsto \ph_t(\eta)$) \[
    \ee^{-N\ph_t(\eta)}\pa{   \partial_t + \mathcal L } \br{ \ee^{N\ph_t} } (\eta) = \eta\Gamma_{\eta}^{\rm exp}\left(\frac{\delta \ph_t}{ \delta \eta}\right) + \frac{1}{N} R^{\rm exp}\br{\ph_t}(\eta) .
    \]

\item The term $\Gamma_{\eta}^{\rm exp}\left(\frac{\delta \ph}{ \delta \eta}\right)$ is the order $0$ (in $N$) term that does not depend on $N$. It comes from the mean-field structure only. Very importantly in order to obtain our exponential concentration result, the exponential carr\'e-du-champs gives back the usual carr\'e-du-champs for small $\ph$ ($\Gamma_{\eta}^{\rm exp}\left(\frac{\delta \ph}{ \delta \eta}\right) \simeq \Gamma_{L_\eta}\left(\frac{\delta \ph}{ \delta \eta}\right) $ when $\ph \to 0$) so that it scales \emph{quadratically} in $\ph$ for $\ph$ small.
\item The term $R^{\rm exp}_t\br{\ph }(\eta)$ depends on $N$ but is of order $\mathcal O_N(1)$ and also comes from the mean-field structure. \item Again, the proof is carried out when $\ee^{N\ph}$ belongs to the domain of the generator $\mathcal L$ of the underlying process (which follows from Assumption~\ref{assu:domain}, but the result certainly still holds in general by replacing the terms of the right hand side in~\eqref{eq:mart-exp-with-rest} by their bounds and stating that $M^{\rm exp,T}$ is a super-martingale.
\end{itemize}
\end{rem}

The proofs of those two propositions will be the subject of section~\ref{sec:weak-back-error-anal}. Using both, we will prove Theorem~\ref{thm:var-lower-bound-h} in Section~\ref{s-sec:proof-thm2-1} in the case of the bias and the variance, in Section~\ref{s-sec:proof-thm2-2} for the $L^p$ error, and Theorem~\ref{thm:exp-bound} in Section~\ref{s-sec:proof-thm3}, by providing bounds for the rests $R$ and $R^{\rm exp}$ in Section~\ref{sec:bound-error}.

\subsection{A general bound on the rest terms}\label{sec:bound-error}

This section is dedicated to the proof of a general bound on the already provided bounds on $R$, $R^{\rm exp}$ and $\Gamma^{\rm exp}$, using the convergences from Assumption~\ref{assu:general}. This general bound in given in the following lemma.

\begin{lem}\label{lem:bound-general}
Recall the definition of $\overline f$ from~\eqref{eq:fun-centered}. Under Assumption~\ref{assu:general}, there exists $C>0$ such that for all $\eta\in \mathcal P$, $0\leqslant t\leqslant T$,
\begin{equation*}
\sqrt{\eta \Gamma_{L_\eta}\po Q^V_{t}\po \overline f \pf\pf}\sqrt{\eta \Gamma_{L_\eta}\po Q^V_{t}(\one)\pf} + \Phi_{t}(\eta)\po \overline f \pf\Gamma_{L_\eta}\po Q^V_{t}(\one) \pf \leqslant Ce^{-2\lambda t}\po w_1(t)\mathcal N_1(f) + w_2(t)\mathcal N_2(f) \pf,
\end{equation*}
and
\begin{equation*}
\eta \Gamma_{L_\eta}\po Q^V_{t}\po \overline f \pf\pf + \Phi_{t}(\eta)\po \overline f \pf^2\eta\Gamma_{L_\eta}\po Q^V_{t}(\one) \pf \leqslant Ce^{-2\lambda t} \po w_1(t)^2\mathcal N_1(f)^2 + w_2(t)^2\mathcal N_2(f)^2 \pf.
\end{equation*}
\end{lem}

\begin{proof}
First notice that, using Condition~\eqref{eq:borne-semi-group} and~\eqref{eq:conv-semi-group} of Assumption~\ref{assu:general}, we have
\begin{equation}\label{eq:conv-non-linear-flow}
\left|\Phi_{t}(\eta)(f) - \eta_{\infty}(f) \right| = \po e^{\lambda t}\eta Q^V_{t}(\one)\pf^{-1} \left| e^{\lambda t}\eta Q^V_t\po f-\eta_\infty(f) \pf\right| \leqslant Cw_1(t)\mathcal N_1(f).
\end{equation}
We decompose the \textit{carr\'e du champ} as:
\[
\Gamma_{L_{\eta}}(Q^V_{t}\po \overline f \pf) = \Gamma_L(Q^V_{t}\po \overline f \pf) + \Gamma_{S_{\eta}}(Q^V_{t}\po \overline f \pf).
\]
The \textit{carr\'e du champ} is a bilinear form. Hence, for $G \in \left\{L,S_\eta\right\}$, Minskowski inequality~\eqref{eq:minskovki} yields
\begin{multline*}
\Gamma_G\po Q^V_{t}\po \overline f \pf\pf = \Gamma_G\left[ Q^V_{t}(f - \eta_{\infty}(f)) + \po\eta_{\infty}(f) - \Phi_{t}(\eta_0^N)(f)\pf Q^V_{t}(\one)\right] \\\leqslant 2\Gamma_G\po Q^V_{t}(f - \eta_{\infty}(f))\pf + 2\po \Phi_{t}(\eta_0^N)(f) - \eta_{\infty}(f)\pf^2 \Gamma_G\po Q^V_{t}(\one)\pf.
\end{multline*}
First, Condition~\eqref{eq:born-carre-champ} gives
\begin{equation*}
\Gamma_L\po Q^V_{t}(f - \eta_{\infty}(f))\pf \leqslant e^{-2\lambda t}w_2(t)^2\mathcal N_2(f)^2,
\end{equation*}
and Inequality~\eqref{eq:conv-non-linear-flow} yields that there exists $C>0$ such that for all $f\in \mathcal A$
\[
\po \Phi_{t}(\eta_0^N)(f) - \eta_{\infty}(f)\pf^2 \leqslant Cw_1(t)^2 \mathcal N_1(f)^2.
\]
The last two inequalities put together with Condition~\eqref{eq:born-carre-champ2} yield that there exists $C>0$ such that
\begin{equation}\label{eq:borne-carre-champ-M}
\Gamma_L(Q^V_{t}\po \overline f \pf) \leqslant Ce^{-2\lambda t} \left[ w_1(t)^2 \mathcal N_1(f)^2 + w_2(t)^2\mathcal N_2(f)^2  \right].
\end{equation}

To get a bound on $\Gamma_{S_\eta}(Q^V_{T-s}\po \overline f \pf)$, we use the bound on the jump kernel~\eqref{eq:bound-jump-kernel}, and write
\begin{align*}
\eta\Gamma_{S_{\eta}}(Q^V_{t}(f-\eta_\infty(f))) &= \int_{E^2} \po Q^V_{t}(f-\eta_\infty(f))(x')-Q^V_{t}(f-\eta_\infty(f))(x)\pf^2 K_\eta(x,\dd x')\eta(\dd x)  \\ &\leqslant 4K_\ast e^{-2\lambda t} \sup_{E} \left| e^{\lambda t}Q^V_t(f- \eta_\infty(f)) \right|^2\\ &\leqslant 4K_\ast e^{-2\lambda t}w_1(t)^2\mathcal N_1(f)^2,
\end{align*}
where we used Condition~\eqref{eq:conv-semi-group} in the last inequality. Similarly, Condition~\eqref{eq:borne-semi-group} yields that
\[
\eta\Gamma_{S_\eta}\po Q^V_{t}(\one)\pf \leqslant Ce^{-2\lambda t},
\]
so that
\begin{equation*}
\eta\Gamma_{S_\eta}(Q^V_{t}\po \overline f \pf) \leqslant Ce^{-2\lambda t}w_1(t)^2 \mathcal N_1(f)^2,
\end{equation*}
and finally
\begin{equation*}
\eta\Gamma_\eta(Q^V_{t}\po \overline f \pf) \\\leqslant Ce^{-2\lambda t}  w_1(t)^2 \mathcal N_1(f)^2 + Ce^{-2\lambda t}w_2(t)^2\mathcal N_2(f)^2.
\end{equation*}
Additionally, we have
\begin{multline*}
\left|\Phi_{t}(\eta)\po \overline f \pf\right| \leqslant \left| \Phi_{t}(\eta)(f) - \eta_\infty(f)\right| + \left| \Phi_{T}(\eta^N_0)(f) - \eta_\infty(f)\right| \\\ \leqslant C\po w_1(t) + w_1(T) \pf\po \mathcal N_1(f)\pf \leqslant 2Cw_1(t)\po \mathcal N_1(f)\pf.
\end{multline*}
Putting everything together concludes the proof of both inequalities.
\end{proof}

\subsection{Proof of Theorem~\ref{thm:var-lower-bound-h}: the case \texorpdfstring{$p=1,2$}{p12}}\label{s-sec:proof-thm2-1}

We are now in position to prove Theorem~\ref{thm:var-lower-bound-h}, by providing an upper bound for the rest 
\[\int_0^T R_{s}\br{\ph_s } \pa{\eta^N_s} \dd s,\] 
which is independent of $T$ for the bias, and one for the quadratic variation for the variance.

\begin{proof}[Proof of Theorem~\ref{thm:var-lower-bound-h} (bias case $p=1$)]
Fix $f\in \mathcal A$, $T>0$, and recall that 
\[
\ph_t(\eta) = \po \Phi_{T-t}(\eta)(f) - \Phi_{T}(\eta_0^N)(f) \pf,
\]
and the definition of $M^{1,T}$ from Equation~\eqref{eq:mart-var-with-rest}.
Because $M^{1,T}$ is a martingale, $\E\po M_T^{1,T}\pf =0$, and thus
\begin{equation*}
\E\po  \eta^N_T(f) - \Phi_{T}(\eta_0^N)(f) \pf = \E\po  \frac{1}{N} \int_0^T \eta_s^N R_{s}\br{\ph_s} \pa{\eta^N_s} \dd s \pf.
\end{equation*}
Under Assumption~\ref{assu:general}, a triangular inequality, proposition~\ref{prop:mart-var} and Lemma~\ref{lem:bound-general} then yield
\begin{align*}
&\left|\E\po \eta^N_T(f) - \Phi_{T}(\eta_0^N)(f) \pf\right| \leqslant \frac{1}{N} \int_0^T \E\left| R_{s}\br{\ph_s } \pa{\eta^N_s} \right| \dd s \\ &\leqslant \frac{C}{N} \int_0^T e^{2\lambda(T-s)} \sup_{\eta\in\mathcal P}\po \sqrt{\eta \Gamma_{L_\eta}\po Q^V_{T-s}\po \overline f \pf\pf}\sqrt{\eta \Gamma_{L_\eta}\po Q^V_{T-s}(\one)\pf} + \Phi_{T-s}(\eta)\po \overline f \pf\Gamma_{L_\eta}\po Q^V_{T-s}(\one) \pf\pf  \dd s \\& \leqslant \frac{C}{N} \po \|w_1\|_1 + \|w_2\|_1 \pf \po \mathcal N_1(f) + \mathcal N_2(f)  \pf,
\end{align*}
which concludes the bias case $p=1$. 
\end{proof}

\begin{proof}[Proof of Theorem~\ref{thm:var-lower-bound-h} (Variance case $p=2$)]
Proposition~\ref{prop:mart-var} yields that $M^{1,T}$ is a martingale with quadratic variation $\left< M^{1,T} \right>$ satisfying Inequality~\eqref{eq:borne-var-quad}. In particular, under Assumptions~\ref{assu:general}, we have
\[
\left< M^{1,T} \right>_T \leqslant \frac{C}{N} \int_0^t e^{2\lambda (T-s)} \po  \eta_s^N \Gamma_{L_{\eta^N_s}}\po Q^V_{T-s}\po \overline f \pf\pf + \Phi_{T-s}(\eta_s^N)\po \overline f \pf^2\eta_s^N\Gamma_{L_{\eta^N_s}}\po Q^V_{T-s}(\one) \pf \pf \dd s,
\]
for some $C>0$.
Lemma~\ref{lem:bound-general} then yields that
\begin{multline*}
\left< M^{1,T} \right>_T \leqslant \frac{C}{N} \int_0^T \po w_1(T-s)^2\mathcal N_1(f)^2 + w_2(T-s)^2\mathcal N_2(f)^2 \pf \dd s \\ \leqslant \frac{C'}{N}\po \|w_1\|_{L^2} + \|w_2\|_{L^2} \pf \po \mathcal N_1(f)^2 + \mathcal N_2(f)^2  \pf ,
\end{multline*}
where $C,C'$ are independent of $T>0$. Then,
\[
\E\po \left| M_T^{T}\right|^{2} \pf = \E\left< M^{1,T} \right>_T \leqslant \frac{C}{N}\po \mathcal N_1(f)^2 + \mathcal N_2(f)^2  \pf.
\]
Finally, we have for all $T>0$:
\begin{align*}
\E \po \po \eta^N_T(f) - \Phi_{T}(\eta_0^N)(f) \pf^2 \pf &= \E\po \po M_T^{1,T} + \frac{1}{N}\int_0^T R\br{\ph_s} \dd s \pf^{2}\pf \\ &\leqslant 2\po  \E\po \left| M_T^{1,T}\right|^2 \pf + \E\po \frac{1}{N}\int_0^T \left| R\br{\ph_s} \right| \dd s\pf^2\pf \\ &\leqslant \frac{C}{N}\po \mathcal N_1(f)^2 + \mathcal N_2(f)^2 \pf + \frac{C}{N^{2}}\po \mathcal N_1(f)^2 + \mathcal N_2(f)^2  \pf
\end{align*}
which concludes the proof.
\end{proof}

\subsection{Proof of Theorem~\ref{thm:var-lower-bound-h}: the case \texorpdfstring{$p>2$}{p3}}\label{s-sec:proof-thm2-2}

In order to treat moments of order $p > 2$, the idea is to use the Burkholder-Davies-Gundy (BDG) inequality. However, since our estimates are controlling the predictable quadratic variation of the main martingale $M^{1,T}_t$ and not its uncompensated quadratic variations, which is the quantity of interest in the BDG inequality, denoted as usual
\[
t \mapsto \br{M^{1,T}}_t
\]
some more analysis about jump sizes are required. Denote by 
\[
\Delta_t X = X_{t}-X_{t^-}\]
the jump at time $t$ of a \textit{c\`adl\`ag} semi-martingale $X$. We need two additional lemmas. The first one is a weaker form of a classical property called "quasi-left continuity" (continuity at any so-called predictable stopping time) of a \textit{c\`adl\`ag} martingale, which is equivalent to the continuity of its predictable quadratic variation.

\begin{lem}\label{lem:quasi-left} 
Let $X$ be a Markov process with generator $L$ and $f\in \mathcal B(E)$ such that
\[
t\mapsto\Gamma_L(Q_{T-t}^V(f))
\]
is well-defined and bounded. Then for any fixed $t_0 \in [0,T]$, the \textit{c\`adl\`ag} version of $Q_{T-t}^V(f))(X_t)$ does not jump at $t_0$ almost surely:
\[
\mathbb P \br{  \lim_{t \to t_0-} Q_{T-t}^V(f)(X_t) =  \lim_{t \to t_0+} Q_{T-t}^V(f)(X_t) } = 1.
\]
\end{lem}

\begin{proof} This is a simplified version of the usual argument of stochastic calculus that derive 'quasi-left continuity' of increasing processes from the continuity of their predictable compensator.

On the one hand, by assumption, the process $t \mapsto \int_0^t \Gamma_L(Q_{T-s}^V(f)(X_s) \dd s $ is well-defined and continuous, and is the predictable quadratic variation of the martingale 
\[
t \mapsto  M^T_t := \ee^{- \int_0^tV(X_s) \dd s}Q_{T-t}^V(f)(X_t).
\]

On the other hand, by definition of the (uncompensated) quadratic variation, the square of the jumps $(\Delta M^T_t)^2$ of $M^T_t$ are exactly given by the jumps of $t \mapsto \br{ M^T}_t$:
\[
(\Delta M^T_t)^2 = \Delta \br{M^T}_t .
\]
However, $\br{ M^T} - \langle M^T \rangle$ is a martingale (by definition), and since $\langle M^T \rangle$ is continuous from our first point, for any $t\in[0,T]$, it holds
\[
\E \po\pa{ \Delta M^T_t } ^2 \pf = \E \pa{ \Delta \br{M^T}_t   } = \E \pa{ \Delta \langle M^T \rangle _t   } = 0.
\]
This concludes the proof.
\end{proof}

Note that the martingale $M^T$ defined in the proof of Lemma~\ref{lem:quasi-left} and its uncompensated quadratic variations may very well have jumps, albeit with jump times whose distribution is atomless in time.

A elementary variant of BDG inequality for the predictable quadratic variation is given by the following. Remark that this is the usual inequality for $p=2$ because quadratic variations have all the same average by definition.

\begin{lem}\label{lem:BDG} 
Let $M$ be any local martingale with jumps bounded by $ m < + \infty $. There exists a universal constant $C_p$ which only depends on $p$, such that for any stopping time $\tau$:
\[
\E \pa{ \sup_{t \leqslant \tau} \abs{M_t}^p  }\leqslant C_p m^p + C_p \E \pa{ \langle M \rangle_\tau ^{p/2} }.
\]
\end{lem}

\add{No version as simple as this one seems to be treated for its own sake in the classical literature so give a proof for completeness.}

\begin{proof}
We are going to show that 
\[
\E \pa{ \br{ M }_\tau ^{p/2} }  \leqslant C_p m^p + C_p \E \pa{ \langle M \rangle_\tau ^{p/2} },
\]
and the result will follow from the usual BDG inequality. 

Consider the martingale $N := \br{ M } -  \langle M \rangle$, and let us show that for all $t\geqslant 0$
\begin{equation}\label{eq:ineq-crochet}
\br{N}_t \leqslant \langle M \rangle^2_t + \sum_{0 \leqslant s \leqslant t} \pa{\Delta_s M }^4.
\end{equation}
Indeed, $N$ has finite variation, and hence its quadratic variation is equal to the sum of the squares of its jumps so that
\begin{align*}
\br{N}_t = \sum_{0 \leqslant s \leqslant t} \pa{\Delta_s N }^2  =   \sum_{s \leqslant t} \pa{\Delta_s \br{ M } -  \Delta_s \langle M \rangle }^2 \leqslant  \sum_{0 \leqslant s \leqslant t} \pa{\Delta_s \br{ M }}^2 + \pa{  \Delta_s \langle M \rangle }^2.
\end{align*}
Moreover, since $\langle M \rangle$ is increasing and $\langle M \rangle_0 = 0$, its jumps are almost surely non-negative and are smaller than its future value, which yields
\[
 \sum_{0 \leqslant s \leqslant t} \pa{\Delta_s \langle M \rangle }^2 \leqslant  \langle M \rangle_t^2.
\]
Finally, by definition of the (uncompensated) quadratic variation, the square of the jumps of a semi-martingales are equal to the jumps of its quadratic variation, that is $\Delta_s \br{M} = (\Delta_s M )^2$, which concludes the proof of~\eqref{eq:ineq-crochet}.

We can now write:
\[
\E \pa{ \br{ M }_\tau ^{p/2} }  \leqslant C_p \E \pa{ \langle M \rangle_\tau ^{p/2} } + C_p  \E \pa{ \abs{N}_\tau^{p/2}},
\]
and then apply the usual BDG inequality to $\E \pa{ \abs{N}_\tau^{p/2}}$ (this requires $p/2 \geqslant 1$) to get
\[
\E \pa{ \abs{N_\tau} ^{p/2} }  \leqslant C_p \E \pa{ \br{N}_\tau^{p/4} } \leqslant C_p \E \pa{ \langle M \rangle_\tau^{p/2} } + C_p \E\po\pa{ \sum_{s \leqslant t} \pa{\Delta_s M }^4 }^{p/4}\pf.
\]
Using young inequality and the fact that $t\mapsto \br{ M }_t$ is increasing, the last term in this inequality can be bounded for any $\eta >0$ by
\begin{multline*}
\E\po\pa{ \sum_{s \leqslant t} \pa{\Delta_s M }^4 }^{p/4}\pf \leqslant m^{p/2} \E\po\pa{ \sum_{s \leqslant t} \pa{\Delta_s M }^2 }^{p/4}\pf =m^{p/2} \E\po\pa{ \sum_{s \leqslant t} \Delta_s \br{ M } }^{p/4}\pf \\ \leqslant \frac{1}{2\eta}m^p +\frac{\eta}{2}  \E\po \po \br{ M }_t \pf^{p/2} \pf,
\end{multline*}
so that, taking $\eta$ small enough, the result follows.
\end{proof}

We can now finish the proof of Theorem~\ref{thm:var-lower-bound-h}.

\begin{proof}[Proof of Theorem~\ref{thm:var-lower-bound-h} (case $p > 2$)]
The proof is similar to the case $p=2$ except we now need to bound uniformly the jumps of the martingale $M^{1,T}_t$. We know from its definition~\eqref{eq:mart-var-with-rest} and from the fact that $t \mapsto Q^V_{T-t}(f)$ is continuous that its jumps are equal to:
\[
\Delta_t M^{1,T} = \lim_{t \to t^+}\frac{\eta^N_{t}Q^V_t(f)}{\eta^N_{t}Q^V_{t}(\one)} - \lim_{t \to t^-}\frac{\eta^N_{t}Q^V_t(f)}{\eta^N_{t}Q^V_{t}(\one)} .
\]
If we condition the particle system at a branching time and consider the particle system up to the next branching time, all particles are independent. We can then consider all the jump times of the process $ t \mapsto Q^V_{T-t}(f)(X^i_t) $ associated with particle $i$ which is almost surely a countable set denoted $\mathcal J^{i}$. Conditionally on these jump times, a particle $j \neq i$ will be independent, so that by Lemma~\ref{lem:quasi-left}, the jumps of the process $t \mapsto Q^V_{T-t}(f)(X^j_t)$ do belong to the set $\mathcal J^{i}$ with probability $0$. As a consequence, almost surely, two particles won't ever jump at the same time (from a death event or not), and
\[
\abs{\Delta_t M^{1,T}} \leqslant \frac{2}{N} \frac{\norm{f}_\infty C_+ ^2}{c_-^2},
\]
where  $C_+,c_- $ are the constant from Condition~\eqref{eq:borne-semi-group}. 

We thus obtain with the modified BDG inequality of Lemma~\ref{lem:BDG}
\[
\E\pa{ \abs{ M^{1,T}_t }^p } \leqslant C_p\frac{\|f\|_{\infty}}{N} + C_p \E\pa{ \langle M^{1,T}_t \rangle^{p/2} },
\]
in which $C_p$ depends only on $p$, $c_-$ and $C_+$. The rest of the proof is then identical to the $p=2$ case.
\end{proof}

\subsection{Proof theorem~\ref{thm:exp-bound}}\label{s-sec:proof-thm3}

The proof of Theorem~\ref{thm:exp-bound} is of similar spirit as the proof of Theorem~\ref{thm:var-lower-bound-h}, but rely on Proposition~\ref{prop:mart-exp} instead of Proposition~\ref{prop:mart-var}.

\begin{proof}[Proof of Theorem~\ref{thm:exp-bound}]
Fix $f\in \mathcal A$, $T>0$, and recall that
\[
\ph_t(\eta) =  \Phi_{T-t}(\eta)(f) - \Phi_{T}(\eta_0^N)(f) ,\qquad \overline f = f - \Phi_{T}(\eta_0)(f),
\]
and the definition of $M^{\rm exp,T}$ from Equation~\eqref{eq:mart-exp-with-rest}. Proposition~\ref{prop:mart-exp} yields that, under Assumption~\ref{assu:general}, $M^{\rm exp,T}$ is a martingale, so that for all $t\geqslant 0$, $\E\po M_T^{\rm exp,T}\pf = 1$, and thus
\[
\E\po\exp\po N\ph_t(\eta_T^N) \pf\pf  \leqslant \exp \po  N\int_0^T \sup_{\eta\in \mathcal P_N}\left[ \eta \Gamma_{\eta}^{\rm exp}\left(\frac{\delta \ph_s}{ \delta \eta}\right) + \frac{1}{N} R^{\rm exp}_{s}\br{\ph_s} \po \eta \pf \right]\dd s\pf.
\]
Using the bound on $\Gamma_{\eta}^{\rm exp}$ and $R^{\rm exp}$ from Proposition~\ref{prop:mart-exp}, as well as Lemma~\ref{lem:bound-general}, we get
\begin{multline*}
\sup_{T\geqslant 0}\int_0^T \sup_{\eta\in \mathcal P_N}\left[ \eta \Gamma_{\eta}^{\rm exp}\left(\frac{\delta \ph_s}{ \delta \eta}\right) + \frac{1}{N} \eta R^{\rm exp}\br{\ph_s}\po \eta\pf \right]\dd s \\ \leqslant C \po \|w_1\|_2^2 + \|w_2\|_2^2\pf \po \mathcal N_1(f)^2 + \mathcal N_2(f)^2 + \|f\|_{\infty}^2 \pf e^{C\|f\|_{\infty}} \\ + \frac{C}{N}\po \|w_1\|_1 + \|w_2\|_1\pf \po \mathcal N_1(f) + \mathcal N_2(f) + \|f\|_{\infty} \pf e^{C\|f\|_{\infty}} ,
\end{multline*}
for some $C>0$, which concludes the proof. 
\end{proof}

\section{Proof of Proposition~\ref{prop:mart-var} and~\ref{prop:mart-exp}}\label{sec:weak-back-error-anal}

This section is dedicated to the proofs of Proposition~\ref{prop:mart-var} and~\ref{prop:mart-exp}. The first one does not require any assumptions. However, for the bound on the exponential martingale, it is necessary to have a uniform lower bound on the survival rates $Q^V_t(\one)$ as given by condition~\eqref{eq:borne-semi-group} of Assumption~\ref{assu:general}. Using the Markov property, we have that for all $N\in\N$, and any time dependent test function $\ph:\R_+\times\mathcal P_N$, the processes $t\mapsto  M_t^v \in \R$ and $t\mapsto  M_t^e \in \R$ defined by:
\begin{equation}\label{eq:mart-var}
  M^v_t := \ph_t(\eta^N_t) - \int_0^t \po \partial_s+\mathcal L \pf (\ph_s) (\eta^N_s) \, \dd s,
\end{equation}
and
\begin{equation}\label{eq:mart-exp}
  M^{e}_t := \exp \po N \ph_t(\eta_{t}^{N}) - \int_{0}^{t}{\rm e}^{ -N\ph_s } (\partial_{s}+\mathcal L) ({\rm e}^{ N\ph_s }) (\eta_{s}^{N}) \dd s \pf,
\end{equation}
are martingales with respect to the filtration at hand. The goal here is then to bound uniformly in $N\in \N$ thanks to the mean-field structure the terms
\[
N\po \partial_t+\mathcal L \pf (\ph_t) (\eta^N_t),\qquad N\po \partial_s+\mathcal L \pf (\ph^2_t) (\eta^N_t),\qquad\text{and}\qquad  \frac1N {\rm e}^{ -N\ph_t } (\partial_{t}+\mathcal L) ({\rm e}^{ N\ph_t }) (\eta_{t}^{N}),
\]
for the specific functions $\ph_t(\eta) := \po \Phi_{T-t}(\eta)\po \overline f \pf\pf$. This will be done with the estimates appearing in Proposition~\ref{prop:mart-var} and~\ref{prop:mart-exp}.

To this end, we need a notion of differential calculus for functions defined on the set of probability measure, namely we want to differentiate $\eta\mapsto \ph(\eta)$ for $\ph\in\mathcal D$. We introduce in Section~\ref{s-sec:derivative} the notion of flat and discrete derivative. Proposition~\ref{prop:mart-var} and~\ref{prop:mart-exp} will then be proven in Section~\ref{s-sec:mart-var} and~\ref{s-sec:mart-exp} respectively. For the remainder of this work, fix $f\in \mathcal A$, and write, for $k\in\left\{1,2\right\}$
\[
\ph_t(\eta) := \po \Phi_{T-t}(\eta)\po \overline f \pf\pf^k, \qquad \overline f : = f - \Phi_{T}(\eta_0^N)(f),
\]
and for a fixed $T$ and the given $\ph$, we denote the martingale given in~\eqref{eq:mart-var} by $M^{1,T}$, and the one from~\eqref{eq:mart-exp} by $M^{\rm exp,T}$.

\subsection{Flat and discrete  derivative}\label{s-sec:derivative}

Let us define more rigorously the notion of flat derivative of measure-valued functions. We restrict to functions defined on probabilities (denoted $\varphi:\mathcal P\to \R$). 

\begin{defi}\label{def:deriv}
We say that $\ph:\mathcal P\to \add{\R^p}$ is differentiable at $\eta_0 \in \mathcal P$ on the convex interval $
I := ( \eta_h )_{h \in [0,1] } $
where 
$\eta_h := h \eta_1 +(1-h)\eta_0$ with $\eta_0,\eta_1\in\mathcal P$ if there exists a measurable function denoted 
\[
(\eta,x) \in I \times E \mapsto \frac{\delta \ph}{\delta \eta}(\eta,x) \in \R
\]
such that for all $h_1 \in [0,1]$
\[
\ph(\eta_{h_1}) - \ph(\eta_{0}) = \int_{0}^{h_1} \int_{E} \frac{\delta \ph}{\delta \eta}\po\eta_h,x\pf \,  (\eta_1-\eta_0)(\dd x) \, \dd h,
\]
and which satisfies the integrability condition:
\[
\int_{0}^{1} \int_{E} \abs{\frac{\delta \ph}{\delta \eta}}\pa{(1-h)\eta_0 + h\eta_1, x} \,  (\eta_1+\eta_0)(\dd x) \, \dd h < +\infty .
\]
If $\nu$ is zero mass real-valued measure such that $\eta + h \nu$ is a probability for all $h > 0$ small enough, we say that $\ph$ is differentiable at $\eta$ in the direction given $\nu$ if $\ph$ is differentiable on $(\eta + h \nu)_{0\leqslant h \leqslant h_1}$ for some $h_1 > 0$ and we will denote
\[
\frac{\delta \ph}{\delta \eta}(\eta)(\nu) = \int_E \frac{\delta \ph}{\delta \eta}(\eta,x)  \nu(\dd x).
\]

\end{defi}

\begin{rem}
\begin{itemize}
    \item Differentiability in a separating class of directions is required in order to obtain uniqueness of the function $x \mapsto \frac{\delta \ph}{\delta \eta}(\eta_0,x)$ at $\eta_0$ up to an additive constant . 
    
    \item With this definition, $\frac{\delta \ph}{\delta \eta}(\eta_0,x)$ can be defined only up to an additive constant. This ambiguity is resolved by  asking that the derivative of constants vanish and that the usual derivation chain rules do apply.
    
    \item  The derivative $\frac{\delta \ph}{\delta \eta}$ is a function of both $\eta\in\mathcal P$ and $x\in E$, but we are going to use the simplified notations, for any generator $G$, 
    \[\eta G\po \frac{\delta \ph}{\delta \eta}\pf = \eta \po G\po x\mapsto \frac{\delta \ph}{\delta \eta}(\eta,x)\pf\pf,\qquad \eta \Gamma_G\po \frac{\delta \ph}{\delta \eta}\pf = \eta \po \Gamma_G\po x\mapsto \frac{\delta \ph}{\delta \eta}(\eta,x)\pf\pf.\]
\end{itemize}
\end{rem}

The only cases considered here is when $\ph(\eta)$ is an integration with respect to $\eta$ composed with usual smooth functions.
\begin{lem}\label{lem:regu-integration}
If $f:E\to\R^p$ is measurable and bounded, then the map $\phi:\eta \mapsto \eta(f)$ is differentiable, and we set
\begin{equation}\label{eq:deriv_lin}
\frac{\delta \phi}{ \delta \eta}(\eta,x) = f(x) - \eta(f).
\end{equation}
Moreover, if $F: \R^p \to \R$ is differentiable and $\phi_F:\eta\mapsto F(\eta(f))$, we set:
\[
\frac{\delta \phi_F}{ \delta \eta}(\eta,x) = \sum_{i=1}^p D_i\pa{F}(\eta(f))\pa{f^i(x) - \eta(f^i)}.
\]
Those definitions do satisfy the requirements of Definition~\ref{def:deriv}.
\end{lem}
\begin{proof}
By definition.
\end{proof}

We have as a direct consequence

\begin{lem}\label{lem:derivative_flow}
Assume that for a given $T-t \in [0,T]$, $x \mapsto Q^{V}_{T-t}\po \one \pf(x)$ is bounded below away from $0$. Then
$\varphi:\eta \mapsto \Phi_{T-t}(\eta)\po \overline f \pf$ is differentiable, and its derivative is given by
\[
\frac{\delta \varphi }{\delta \eta}\po \eta,x\pf =  \frac{1}{\eta Q^V_{T-t}(\one)}  Q^V_{T-t}\po \overline f \pf(x) - \frac{\eta Q^V_{T-t}\po \overline f \pf}{\eta Q^V_{T-t}(\one)^2}  Q^V_{T-t}(\one)(x).
\]
\end{lem}

\begin{proof}
This is a direct consequence of the definition~\eqref{eq:FK_flow} and Lemma~\ref{lem:regu-integration}.
\end{proof}

The proofs of the present work rely on generator calculus based on the approximation of the mean-field generator $L_\eta$ by a particle system. It turns out that such calculus can be made more transparent with some notion of 'discrete' or 'particle' derivative. 

\begin{defi}\label{def:discrete-derivative}
Let $\eta \in \mathcal P$ be given and assume that $\eta + \frac{1}{N}\nu \in \mathcal P $ where $\nu$ is a signed measure with mean $0$. For all $\ph:\mathcal P \to\R$, we define the \emph{particle (or discrete) derivative} in the direction $\nu$ by 
\[
\frac{\delta_{N} \ph}{ \delta_{N} \eta}(\eta)(\nu) = N \po\ph(\eta + \frac{1}{N}\nu) - \ph(\eta) \pf.
\]
If $\ph$ is moreover differentiable at $\eta$ in the direction $\nu$, the second order particle (or discrete) derivative (in the direction $\nu$) is defined by: 

\begin{equation*}
  \frac{\delta^{2}_{N} \ph}{ {\delta_{N} \eta}^2 }(\eta)(\nu) = 2N \po \frac{\delta_{N} \ph}{ \delta_{N} \eta}(\eta)(\nu) - \nu\po\frac{\delta \ph}{ \delta \eta}(\eta,\cdot)\pf \pf.
\end{equation*}
\end{defi}

\begin{rem} Few remarks.
\begin{itemize}
    \item Notice that if $\ph$ is differentiable in the direction $\nu$, we have
\[
\lim_{N \to +\infty} \frac{\delta_{N} \ph}{ \delta_{N} }(\eta)(\nu) = \frac{\delta \ph}{ \delta \eta}(\eta)(\nu).
\]
    \item Moreover, a second order derivative $\frac{\delta^{2} \ph}{ {\delta \eta}^2 }$ could also be defined by iterating the previous definition, and in this case we would also have
\[
  \lim_{N \to +\infty} \frac{\delta^{2}_{N} \ph}{ {\delta_{N} \eta}^2 }(\eta)(\nu) =  \frac{\delta^{2} \ph}{ {\delta \eta}^2 }(\eta)(\nu,\nu).
\]
\end{itemize}

\end{rem}

We are going to bound the integral part of the martingales~\eqref{eq:mart-var} and~\eqref{eq:mart-exp} using those derivatives, respectively in Section~\ref{s-sec:mart-var} and Section~\ref{s-sec:mart-exp}. To do so, we need bound on the second order discrete derivatives of 
\[
\ph^k_t(\eta)= \po \Phi_{T-t}(\eta)\po \overline f \pf\pf^k,
\]
for all $0\leqslant t\leqslant T$ and $k\in \left\{1,2\right\}$, in the direction $(\delta_{x'}-\delta_x)/N$, $x,x'\in E$.

\begin{lem}\label{lem:disc-derivative-flow}
Let $\eta = \eta_\x \in \mathcal P_N$ be a given empirical probability of sample size $N$ and assume that \[ \nu = \pa{\delta_{x'}-\delta_{x^i}} \] with $x' \in E$ and $x^i \in \x$. We recall the notation $\eta^{(i)} = \frac1N \sum_{j \neq i}\delta_{x^j}$ which is a non-negative measure of mass $\frac{N-1}{N}$. It holds:

\begin{equation*}
    \left|\frac{\delta^2_N \ph_t}{\delta_N \eta^2}\po \eta\pf\po \nu \pf \right| \leqslant \frac{2}{ \po\eta^{(i)} Q^V_{T-t}(\one)\pf^2} \po \left|\nu Q^V_{T-t} \po \overline f \pf \nu Q^V_{T-t} \po \one \pf \right| + \pa{\Phi_{T-t}(\eta)\po \overline f \pf} \pa{\nu Q^V_{T-t}(\one) }^2 \pf,
\end{equation*}
and
\begin{equation*}
    \left|\frac{\delta^2_N \ph^2_t}{\delta_N \eta^2}\po \eta\pf\po \nu \pf \right| \leqslant \frac{1}{ \po\eta^{(i)}  Q^V_{T-t}(\one)\pf^2} \po   6\pa{\nu Q^V_{T-t} \po \overline f\pf }^2 + 10 \pa{\Phi_{T-t}(\eta)\po \overline f \pf}^2 \pa{\nu Q^V_{T-t}(\one) }^2 \pf .
\end{equation*}
\end{lem}

\begin{proof}
Let us first show the result for $k=1$. Fix $\eta_{\x} \in \mathcal P_N$, $x' \in E$, $x^i \in \x$ and denote
\[
\tilde \eta := \eta + \nu/N = \frac1N \delta_{x'} +\frac1N \sum_{j \neq i} \delta_{x^j} .
\]
Fix some non-negative kernel $P:\mathcal P\to \mathcal P$ and write
\[
\Phi(\eta) = \frac{\eta P\po \overline f \pf}{ \eta P\po \one \pf} .
\]
From the definition of the discrete derivative, we get
\begin{multline*}
\frac{\delta_N \Phi}{\delta_N\eta} (\eta) \pa{\nu} = N\po \frac{\tilde \eta P\po \overline f \pf}{ \tilde \eta P\po \one \pf} - \frac{\eta P\po \overline f \pf}{ \eta P\po \one \pf} \pf \\ = \frac{1}{\eta P\po \one \pf \tilde \eta P\po \one \pf} \po \po \tilde \eta - \eta \pf P\po \overline f \pf  \eta P\po \one \pf - \po \tilde \eta - \eta\pf P\po \one \pf  \eta P\po \overline f \pf  \pf \\= \frac{1}{\eta P\po \one \pf \tilde \eta P\po \one \pf} \po \nu P\po \overline f \pf  \eta P\po \one \pf - \nu P\po \one \pf  \eta P\po \overline f \pf  \pf.
\end{multline*}
Lemma~\ref{lem:regu-integration} yields that $\Phi$ is differentiable, and
\[
\frac{\delta\Phi}{\delta\eta}(\eta,x) = \frac{1}{ \eta P\po \one \pf } P\po \overline f \pf(x) - \frac{\eta P\po \overline f \pf}{ \po\eta P\po \one \pf\pf^2 }P\po \one \pf(x).
\]
The last two equalities put together with the definition of the second order discrete derivative yield
\begin{multline*}
   \frac{\delta_N^2\Phi}{\delta_N\eta^2} \po \eta \pf \pa{\nu} 
= 2N \po \frac{\delta_{N} \Phi}{ \delta_{N} \eta}(\eta)(\nu) - \frac{\delta \Phi}{ \delta \eta}(\eta)(\nu) \pf \\ = \frac{2}{\eta P\po \one \pf \tilde \eta P\po \one \pf} \br{-\nu P\po \one \pf \, \nu P\po \overline f \pf + \frac{\eta P\po \overline f \pf}{ \eta P\po \one \pf } \pa{\nu P\po \one \pf }^2 },
\end{multline*}
so that, since $\tilde \eta\po g \pf \geqslant \eta^{(i)}\po g \pf$ and $\eta\po g \pf \geqslant \eta^{(i)}\po g \pf$ for all $g\geqslant 0$, it holds (with $g= P\po\one\pf$):
\begin{equation}\label{eq:second-discret-derivative-1}
\frac{\delta_N^2\Phi}{\delta_N\eta^2} \po \eta \pf \pa{\nu} \leqslant \frac{2}{\pa{\eta^{(i)} P\po \one \pf}^2} \po \left|\nu P\po \one \pf \, \nu P\po \overline f \pf\right| + \Phi(\eta) \pa{\nu P\po \one \pf }^2 \pf.
\end{equation}
Plugging $P= Q_{T-t}^V$ in the previous inequality concludes in the case $k=1$. For $k=2$, we first want to prove the following chain rule for squares:
\begin{equation}\label{eq:chain-rule}
\frac{\delta_N^2\pa{ \Phi^2}}{\delta_N\eta^2} = 2 \Phi \frac{\delta_N^2\Phi}{\delta_N\eta^2} + 2 \pa{\frac{\delta_N\Phi}{\delta_N\eta}}^2.
\end{equation}
Write
\[
\frac{\delta_N\pa{ \Phi^2}}{\delta_N\eta}(\eta)(\nu) = N\po \Phi^2(\eta + \frac{1}{N}\nu ) - \Phi^2(\eta)\pf = \frac{\delta_N \Phi}{\delta_N\eta}(\eta)(\nu) \po \Phi(\eta + \frac{1}{N}\nu) + \Phi(\eta) \pf.
\]
From Lemma~\ref{lem:regu-integration}, we also have that
\[
\frac{\delta\pa{ \Phi^2}}{\delta\eta} = 2\Phi \frac{\delta \Phi}{\delta\eta},
\]
so that
\begin{align*}
\frac{\delta_N^2\pa{ \Phi^2}}{\delta_N\eta^2}(\eta)(\nu) &= 2N\po \frac{\delta_N\Phi}{\delta_N\eta}(\eta)(\nu) \po \Phi(\eta + \frac{1}{N}\nu) + \Phi(\eta) \pf  - 2\Phi(\eta)\frac{\delta\Phi}{\delta\eta}(\eta)(\nu) \pf \\ &= 2N\frac{\delta_N\Phi}{\delta_N\eta}(\eta)(\nu) \po \Phi(\eta + \frac{1}{N}\nu) - \Phi(\eta) \pf + 4N\Phi(\eta)\po \frac{\delta_N \Phi}{\delta_N\eta}(\eta)(\nu) - \frac{\delta\Phi}{\delta\eta}(\eta)(\nu) \pf \\ &= 2\po \frac{\delta_N\Phi}{\delta_N\eta}(\eta)(\nu) \pf^2 + 2\Phi(\eta)\frac{\delta_N^2\Phi}{\delta_N\eta^2}(\eta)(\nu),
\end{align*}
which proves Equation~\eqref{eq:chain-rule}. Now write
\begin{equation*}
\pa{\frac{\delta_N\Phi}{\delta_N\eta}}^2 = \frac{1}{\po \tilde \eta \po P\po \one \pf\pf\pf^2}\po \nu P\po \overline f \pf - \Phi(\eta)\nu P\po \one \pf \pf^2 \leqslant \frac{2}{\po \eta^{(i)} P\po \one \pf\pf^2} \po \po \nu P\po \overline f \pf \pf^2 + \Phi(\eta)^2 \po \nu P\po \one \pf \pf^2 \pf. 
\end{equation*}
Inequality~\eqref{eq:second-discret-derivative-1} and inequality $2ab\leqslant a^2 + b^2$  yield
\[
\Phi \frac{\delta_N^2\Phi}{\delta_N\eta^2} \leqslant \frac{1}{\po \eta^{(i)} P\po \one \pf\pf^2} \po \po  \nu P\po \overline f \pf \pf^2 + 3\Phi(\eta)^2 \po \nu P\po \one \pf \pf^2 \pf,
\]
so that the chain rule~\eqref{eq:chain-rule}, yields
\[
\frac{\delta_N^2\pa{ \Phi^2}}{\delta_N\eta^2} \leqslant \frac{1}{\po \eta^{(i)} P\po \one \pf\pf^2} \po 6 \po \nu P\po \overline f \pf \pf^2 + 10 \Phi(\eta)^2 \po \nu P\po \one \pf \pf^2 \pf,
\]
and we conclude the proof by plugging again $P = Q_{T-t}^V $.
\end{proof}

\subsection{Proof of Proposition~\ref{prop:mart-var}}\label{s-sec:mart-var}

The goal of this section is to prove Proposition~\ref{prop:mart-var}. Let us recall that the process of interest is the semi-martingale $t \mapsto 
\ph_t(\eta^N_t)$ where $ \ph_t(\eta)= \po \Phi_{T-t}(\eta)\po \overline f \pf\pf $, $\Phi_{T-t}$ being the normalized Feynman-Kac non-linear semi-group. From usual stochastic analysis of Markov processes, 
\[
\ph_t(\eta^N_t) = M^{1,T}_t + \int_0^t \pa{\partial_s + \mathcal L}\br{\ph_s}(\eta^N_s) \dd s,  
\]
and the predictable quadratic variation of $M^{1,T}$ is given by
\[
\langle  M^{1,T} \rangle_t = \int_0^t 2 \Gamma_{\mathcal L}\br{\ph_s}(\eta^N_s) \dd s .
\]
The proof consists in computing $\mathcal L\br{\ph_t}$, $\mathcal L\br{\ph_t^2}$ and $\partial_t \ph_t$ and then estimating $\pa{\partial_t + \mathcal L}\br{\ph_t}$ as well as $2 \Gamma_{\mathcal L}( \ph_t) = \mathcal L\br{\ph_s^2} - 2 \ph_t \mathcal L \pa{\ph_t}  $ using the explicit expression of the Feynamn-Kac semi-group.

As a preliminary, recall the definition of the \textit{carr\'e-du-champs} from definition~\ref{def:carre-du-champs}, and let us prove that it can be identified with local variances through the following property.

\begin{lem}\label{lem:carre-du-champ-as-derivatives}
For all $x\in E$, $\eta\in \mathcal P$ and $g$ such that $g,g^2\in \mathcal D\po L_\eta\pf$ ($\eta$ being fixed), it holds that:
\[
\lim_{u\rightarrow 0} \frac{1}{u} \E_x\po \po g(Y_u) - g(x) \pf^2 \pf = 2\Gamma_{L_\eta}\po g\pf(x),
\]
where $Y$ is a Markov process with generator $L_{\eta}$.
\end{lem}

\begin{proof}
Write
\[
\po g(Y_u) - g(x) \pf^2 = \po g(Y_u) \pf^2 + \po g(x) \pf^2 - 2g(Y_u)g(x) = \po g(Y_u) \pf^2 - \po g(x) \pf^2 -2g(x)\po g(Y_u) - g(x) \pf,
\]
so that
\[
\lim_{t\rightarrow 0} \frac{1}{t} \E_x\po \po g(Y_u) - g(x) \pf^2 \pf = L_\eta(g^2)(x) - 2 g(x) L_{\eta}\po g \pf(x) = 2\Gamma_{L_\eta}\po g \pf(x).
\]
\end{proof}

The following lemma is the corner stone of the proof of Proposition~\ref{prop:mart-var}.

\begin{lem}\label{lem:Lcal_gen}   
Under Assumption~\ref{assu:domain}, it holds that
\begin{equation*}
    \mathcal L(\ph^k_t)(\eta) = \eta L_{\eta} \po \frac{\delta \pa{\ph^k_t}}{ \delta \eta} \pf +  \frac{1}{N} R\br{\ph^k_t}(\eta),
  \end{equation*} 
where the rest terms do satisfy:
\begin{multline*}
\left| R\br{\ph_t}(\eta) \right| \leqslant \max_{1\leqslant i\leqslant N}\frac{4}{\po\eta^{(i)} \po Q^V_{T-t}(\one)\pf\pf^2}\\ \times \po \sqrt{\eta \Gamma_{L_\eta}\po Q^V_{T-t}\po \overline f \pf\pf}\sqrt{\eta \Gamma_{L_\eta}\po Q^V_{T-t}(\one)\pf} + \Phi_{T-t}(\eta)\po \overline f \pf \eta \Gamma_{L_\eta}\po Q^V_{T-t}(\one) \pf \pf,
\end{multline*}
as well as
\begin{equation*}
    \left| R\br{\ph^2_t}(\eta)\right| \leqslant \max_{1\leqslant i\leqslant N}\frac{4}{\po\eta^{(i)} \po Q^V_{T-t}(\one)\pf\pf^2}\po  3 \eta \Gamma_{L_\eta}\po Q^V_{T-t}\po \overline f \pf\pf + 5 \Phi_{T-t}(\eta)\po \overline f \pf^2\eta \Gamma_{L_\eta}\po Q^V_{T-t}(\one) \pf \pf .
\end{equation*}
\end{lem}

\begin{proof}[Proof of Lemma~\ref{lem:Lcal_gen}]
Fix $t \geq 0$ and $\eta_\x = \eta \in \mathcal P_N$,
\[
\eta=\eta_{\x}=\frac{1}{N}\sum_{i=1}^N\delta_{x_i},\qquad (x_1,\cdots,x_N)\in E^N,
\]
and for each $i\in \llbracket1,N\rrbracket$, let $Y^i$ be a Markov process with initial condition $Y^i_0 = x_i$ and generator $L_{\eta}$. Since $\eta$ is fixed, $Y^i$ evolves according to the initial generator $L$, and at rate $V$ jumps at the location of one of the $(x_j)_{1\leqslant j \leqslant N}$. For $k\in\left\{1,2\right\}$, we claim that 
\begin{equation}\label{eq:R_expli}
 R\br{\ph^k}(\eta) = 
   \lim_{u \to 0}  \frac1u \frac{1}{2N} \sum_{i = 1}^N  \mathbb E \br{  \frac{\delta^{2}_{N} \ph^k}{ {\delta_{N} \eta}^2 }(\eta)(\delta_{Y^{i}_u} - \delta_{x_i})              }
\end{equation}
Indeed, we have, using the definition of the discrete derivative from Definition~\ref{def:discrete-derivative}:
\begin{align*}
N\left(\ph_t^k(\eta + \frac{1}{N}\po\delta_{Y^i_u}-\delta_{x_i}\pf) - \ph_t^k(\eta)\right)  & = \frac{\delta_{N} \ph_t^k}{ \delta_{N} \eta}(\eta)(\delta_{Y^i_u} - \delta_{x_i}) \pm \po \frac{\delta\ph_t^k}{\delta \eta} (\eta,Y^i_u) - \frac{\delta\ph_t^k}{\delta \eta} (\eta,x_i) \pf \\  & = \frac{\delta\ph_t^k}{\delta \eta} (\eta,Y^i_u) - \frac{\delta\ph_t^k}{\delta \eta} (\eta,x_i) + \frac{1}{2N} \frac{\delta^2_N\ph_t^k}{\delta_N\eta^2}(\eta,\delta_{Y^i_u} - \delta_{x_i}).
\end{align*}
Under Assumption~\ref{assu:domain}, $x_i\mapsto \varphi^k\po\eta^{\x}\pf \in \mathcal D\po L_\eta \pf$, and Lemma~\ref{lem:regu-integration} and~\ref{lem:derivative_flow} yields that $\frac{\delta \ph_t^k}{\delta \eta}\in \mathcal D\po L_\eta \pf$. Hence, we have that 
\begin{equation}\label{eq:proof-Lcal_gen-1}
\lim_{u\rightarrow 0}\frac{1}{u}\E\po\ph_t^k(\eta + \frac{1}{N}\po\delta_{Y^i_u}-\delta_{x_i}\pf) - \ph_t^k(\eta)\pf = L_{\eta}(x_i\mapsto \ph_t^k(\eta_{\x})),
\end{equation}
as well as:
\begin{equation}\label{eq:proof-Lcal_gen-2}
\lim_{u\rightarrow 0}\frac{1}{u}\E\po \frac{\delta\ph_t^k}{\delta \eta} (\eta,Y^i_u) - \frac{\delta\ph_t^k}{\delta \eta}(\eta,x_i) \pf = L_{\eta}\po \frac{\delta \ph_t^k}{\delta \eta} \pf (\eta,x_i).
\end{equation}
This yields~\eqref{eq:R_expli}.

Let's now focus on the case $k=1$. Lemma~\ref{lem:disc-derivative-flow} yields that
\begin{multline*}
\left|\frac{\delta^2_N\ph_t}{\delta_N\eta^2}(\eta,\delta_{Y^i_u} - \delta_{x_i})\right| \\\leqslant \frac{2}{\po\eta^{(i)} \po Q^V_{T-t}(\one)\pf\pf^2}  \left| Q^V_{T-t}\po \overline f \pf(Y^i_u)-Q^V_{T-t}\po \overline f \pf(x_i)\right|\left| Q^V_{T-t}(\one)(Y^i_u)-Q^V_{T-t}(\one)(x_i)\right| \\ + \frac{2}{\po\eta^{(i)} \po Q^V_{T-t}(\one)\pf\pf^2}\Phi_{T-t}(\eta)\po \overline f \pf \pa{Q^V_{T-t}(\one)(Y^i_u)-Q^V_{T-t}(\one)(x_i) }^2 ,
\end{multline*}
so that using Cauchy-Schwartz inequality 
\begin{multline*}
\E\left|\frac{\delta^2_N\ph_t}{\delta_N\eta^2}(\eta,\delta_{Y^i_u} - \delta_{x_i})\right| \\\leqslant \frac{2}{\po\eta^{(i)} \po Q^V_{T-t}(\one)\pf\pf^2} \sqrt{\E\po \left( Q^V_{T-t}\po \overline f \pf(Y^i_u)-Q^V_{T-t}\po \overline f \pf(x_i)\right)^2\pf}\sqrt{\E\po \left( Q^V_{T-t}(\one)(Y^i_u)-Q^V_{T-t}(\one)(x_i)\right)^2\pf} \\ + \frac{2}{\po\eta^{(i)} \po Q^V_{T-t}(\one)\pf\pf^2}\Phi_{T-t}(\eta)\po \overline f \pf \E\po\pa{Q^V_{T-t}(\one)(Y^i_u)-Q^V_{T-t}(\one)(x_i) }^2\pf ,
\end{multline*}
and finally using Lemma~\ref{lem:carre-du-champ-as-derivatives} we get that 
\begin{multline}\label{eq:proof-Lcal_gen-3-k1}
\limsup_{u\rightarrow 0} \frac{1}{u}\left|\E\po \frac{\delta^2_N\ph_t}{\delta_N\eta^2}(\eta,\delta_{Y^i_u} - \delta_{x_i}) \pf  \right| \\\leqslant \frac{4}{\po\eta^{(i)} \po Q^V_{T-t}(\one)\pf\pf^2} \po \sqrt{\Gamma_{L_\eta}\po Q^V_{T-t}\po \overline f \pf \pf (x_i)}\sqrt{\Gamma_{L_\eta}\po Q^V_{T-t}(\one) \pf (x_i)} + \Phi_{T-t}(\eta)\po \overline f \pf \Gamma_{L_\eta}\po Q^V_{T-t}(\one) \pf (x_i) \pf.
\end{multline}
Putting together Equations~\eqref{eq:proof-Lcal_gen-1},~\eqref{eq:proof-Lcal_gen-2} and~\eqref{eq:proof-Lcal_gen-3-k1} yields
\begin{multline*}
\left| N L_{\eta}(x_i\mapsto \ph_t(\eta_{\x})) - L\po \frac{\delta \ph_t}{\delta \eta}(\eta,\cdot) \pf (x_i) \right| \\ \leqslant  \max_{1 \leqslant j \leqslant N}\frac{4}{\po\eta^{(j)} \po Q^V_{T-t}(\one)\pf\pf^2}  \po \sqrt{\Gamma\po Q^V_{T-t}\po \overline f \pf \pf (x_i)}\sqrt{\Gamma\po Q^V_{T-t}(\one) \pf (x_i)} + \Phi_{T-t}(\eta)\po \overline f \pf \Gamma\po Q^V_{T-t}(\one) \pf (x_i) \pf.
\end{multline*}
Recalling that 
\[
\mathcal L (\ph_t) (\eta) = \frac{1}{N} \sum_{i=1}^N N \times L_{\eta}(x_i\mapsto \ph_t(\eta_{\x})),
\]
integrating with respect to $\eta$ and using Cauchy-Schwartz inequality conclude the proof in the case $k=1$.

For the case $k=2$, we have by Lemma~\ref{lem:disc-derivative-flow}
\begin{multline*}
\left|\frac{\delta^2_N\ph_t^2}{\delta_N\eta^2}(\eta,\delta_{Y^i_u} - \delta_{x_i})\right| \leqslant \frac{12}{\po\eta^{(i)} \po Q^V_{T-t}(\one)\pf\pf^2} \pa{ Q^V_{T-t}\po \overline f \pf(Y^i_u)-Q^V_{T-t}\po \overline f \pf(x_i)}^2 \\ + \frac{20}{\po\eta^{(i)} \po Q^V_{T-t}(\one)\pf\pf^2} \pa{\Phi_{T-t}(\eta)\po \overline f \pf}^2 \pa{Q^V_{T-t}(\one)(Y^i_u)-Q^V_{T-t}(\one)(x_i) }^2,
\end{multline*}
and the rest follows as in the case $k=1$.
\end{proof}

Now we are ready to conclude:

\begin{proof}[Proof of Proposition~\ref{prop:mart-var}]
Fix $T>0$, and recall the definition of the non-linear flow from~\eqref{eq:FK_flow}, as well as
\[
\overline f = f - \Phi_T(\eta_0^N)(f).
\]

We first claim that the non-linear semi-group indeed satisfies the Kolmogorov equation~\eqref{eq:Kolmo}, and that 
\begin{equation}\label{eq:evolution2}
\partial_t\ph^k_t(\eta) = -\eta L_\eta\po  \frac{\delta\ph^k_t}{\delta\eta}  \pf ,\qquad k\in\left\{1,2\right\},
\end{equation}
where here
$ \ph^k(t,\eta)= \po \Phi_{T-t}(\eta)\po \overline f \pf\pf^k$. Indeed, under Assumption~\ref{assu:domain}, from the definition of $Q^V$, we have that
\[
\partial_t \eta Q^V_{T-t}\po \overline f \pf = -\eta (L-V) \po Q^V_{T-t}\po \overline f \pf\pf  = -\eta L_\eta \po Q^V_{T-t}\po \overline f \pf\pf+ \eta(V)\eta Q^V_{T-t}\po \overline f \pf,
\]
so that
\begin{align*}\label{eq:evolution1}
\partial_t \varphi_t(\eta) = -\frac{1}{\eta Q^V_t\po \overline f \pf}\eta  L_\eta\po Q^V_t\po \overline f \pf\pf+ \frac{\eta Q^V_t\po \overline f \pf}{\po\eta Q^V_{t}(\one)\pf^2}\eta L_\eta\po Q^V_t\po \overline f \pf\pf = -\eta L_\eta \po \frac{\delta \varphi_t}{\delta\eta} \pf.
\end{align*}
We also have that
\begin{equation*}
\partial_t \varphi_t^2(\eta) = 2 \Phi_{T-t}(\eta)\po \overline f \pf\partial_t \Phi_{T-t}(\eta)\po \overline f \pf =  -\eta  L_{\eta}\po \frac{\delta \varphi_t^2}{\delta\eta}\pf ,
\end{equation*}
so that \eqref{eq:evolution2} holds true.

Now, the expression of the martingale~\eqref{eq:mart-var}, Equation~\eqref{eq:evolution2} and Lemma~\ref{lem:Lcal_gen} yield that
\[
\int_0^t \po \partial_s+\mathcal L \pf (\ph_s^k) (\eta^N_s) \, \dd s = \frac{1}{N}\int_0^t R\br{ \ph^{k}_s}\pa{\eta_s^N} \dd s,
\]
for $k\in\left\{1,2\right\}$.
Hence $M^{1,T}$ is a martingale with the corresponding bound on the rest term. The quadratic variation of the martingale~\eqref{eq:mart-var} is given by
\[
\left< M^{1,T} \right>_t = \int_0^t \po \partial_s + \mathcal L \pf (\ph^2_s)(\eta^N_s) - 2 \ph\po \partial_t + \mathcal L \pf (\ph_s)(\eta^N_s) \dd s,
\]
which yields
\[
\left< M^{1,T} \right>_t = \frac{1}{N}\int_0^t R\br{ \ph^{2}_s}(\eta_s^N ) - 2 \ph_s(\eta_s^N) R\br{ \ph_{s}}(\eta_s^N) \dd s,
\]
and concludes the proof.
\end{proof}

\subsection{Proof of Proposition~\ref{prop:mart-exp}}\label{s-sec:mart-exp}

The goal of this section is to prove Proposition~\ref{prop:mart-exp} using the martingale~\eqref{eq:mart-exp} in the same way as for the proof of Proposition~\ref{prop:mart-var} based on the martingale~\eqref{eq:mart-var}. Let us recall that the latter is the usual exponential martingale
\[
    M^{\exp,T}_t := \exp \pa{ N\ph_t(\eta^N_t) - \int_0^t \ee^{-N\ph_s(\eta)}\pa{ \partial_s + \mathcal L } \br{ \ee^{N\ph_s} } (\eta^N_s) \dd s } .
\]
It is useful, in order to understand the decomposition 
\[
\eta\Gamma_{\eta}^{\rm exp}\left(\frac{\delta \ph_t}{ \delta \eta}\right) + \frac{1}{N} R^{\rm exp}\br{\ph_t}(\eta)
\]
of the proposition, to introduce the 'exponential' generator of the particle system defined for any probability-valued test function $\ph: \mathcal P \to \R$ such that $e^{N\varphi}\in\mathcal D(\mathcal L)$ by
\[
\Ha_{N}(\ph) := \frac{1}{N}{\rm e}^{ -N\ph } \mathcal L \po {\rm e}^{ N\ph } \pf .
\]
When $N$ is large and $\mathcal L$ has a mean-field structure, its formal limit is given by
\[
\lim_{N \to + \infty} \Ha_{N}(\ph) = \eta L_{\eta}\left(\frac{\delta \ph}{ \delta \eta}\right) + \eta \Gamma_{\eta}^{\rm exp}\left(\frac{\delta \ph}{ \delta \eta}\right)
\]
where the second term in the right hand side is the corresponding 'exponential' \textit{carr\'e-du-champs} defined by:
\begin{equation}\label{eq:exp-carre-du-champ}
 \Gamma_{\eta}^{\rm exp}\po f \pf = {\rm e}^{-f}L_{\eta}( {\rm e}^{f})-L_{\eta}(f) \geqslant 0,   
\end{equation}
for any $f\in \mathcal D\po L_\eta \pf$ such that $\ee^f\in \mathcal D\po L_\eta \pf$.  As before, the proof will consists in computing $\Ha_{N}(\ph_t)$ for fixed $N$ as well as $\frac1N \ee^{-N\ph_t} \partial_t \br{ \ee^{N\ph_t} } $ and estimating the rest terms given by their sum.
We have then an equivalent of Lemma~\ref{lem:Lcal_gen} for this exponential generator, that we split in two bounds. First:

\begin{lem}\label{lem:Hcal_gen} Under Assumption~\ref{assu:domain}, it holds: 
\begin{equation}\label{eq:Hcal_gen}
\Ha_{N}(\ph_t)(\eta) =  \eta L_{\eta}\left(\frac{\delta \ph_t}{ \delta \eta}\right) + \eta \Gamma_{\eta}^{\rm exp}\left(\frac{\delta \ph_t}{ \delta \eta}\right)  + \frac{1}{N} \eta R^{\rm exp}\br{\ph_t}(\eta),
\end{equation}
where the rest $R^{\rm exp}$ satisfies, under Condition~\eqref{eq:borne-semi-group} of Assumption~\ref{assu:general}:
\begin{multline*}
    \left|\eta R^{\rm exp}\br{\ph}(\eta)\right| \\\leqslant \max_{1\leqslant i\leqslant N}\frac{2e^{C\|f\|_{\infty}}}{\po\eta^{(i)} \po Q^V_{T-t}(\one)\pf\pf^2} \po \sqrt{\eta \Gamma_{L_\eta}\po Q^V_{T-t}\po \overline f \pf\pf}\sqrt{\eta \Gamma_{L_\eta}\po Q^V_{T-t}(\one)\pf} + \Phi_{T-t}(\eta)\po \overline f \pf\eta\Gamma_{L_\eta}\po Q^V_{T-t}(\one) \pf \pf,
\end{multline*}
for some constant $C > 0$ that depends only on $C_+,c_-$
\end{lem}

Since, contrary to Lemma~\ref{lem:Lcal_gen}, there is the additional term $\eta \Gamma_{\eta}^{\rm exp}\left(\frac{\delta \ph}{ \delta \eta}\right)$, we need the following additional bound:

\begin{lem}\label{lem:bound-flat-derivative}
Under Assumption~\ref{assu:domain}, there exists $C>0$, depending only on $C_+,c_-$, such that the exponential \textit{carr\'e-du-champs} satisfies:
 \begin{multline*}
0 \leqslant \eta \Gamma^{\rm exp}_{\eta}\po \frac{\delta \ph}{ \delta \eta}\pf \leqslant \frac{{\rm e}^{C\|f\|_{\infty}}}{2} \eta \Gamma_{L_\eta}\po \frac{\delta \ph}{ \delta \eta}\pf \\\leqslant \max_{1\leqslant i\leqslant N}\frac{{\rm e}^{C\|f\|_{\infty}}}{\po\eta^{(i)} \po Q^V_{T-t}(\one)\pf\pf^2}\po  \eta \Gamma_{L_\eta}\po Q^V_{T-t}\po \overline f \pf\pf + \Phi_{T-t}(\eta)\po \overline f \pf^2\eta \Gamma_{L_\eta}\po Q^V_{T-t}(\one) \pf \pf.
 \end{multline*}
\end{lem}

Before proceeding to the proof of the above lemmas, we give some bounds on probability-valued derivatives of the main test function:


\begin{lem}\label{lem:bounded-exp-norm-inf}
Under \eqref{eq:borne-semi-group}, for any $t > 0$ and any $x,x'\in E$, it holds:
\[
\left| \frac{\delta \ph_t}{ \delta \eta}(\eta,x) \right| \leqslant  \|f\|_{\infty} \pa{C_+/c_-+C_+^2/c_-^2} ,
\]
as well as:
\[
\left|\frac{\delta^{2}_N \ph_t}{ \delta_N \eta^{2}} (\eta) ( \delta_{x'} - \delta_x) \right| \leqslant 6 \|f\|_{\infty} C_+^2/c_-^2.
\]
\end{lem}

\begin{proof}
This is a direct application of Lemma~\ref{lem:derivative_flow} for the first inequality, and of Lemma~\ref{lem:disc-derivative-flow} for the second.
\end{proof}

\begin{proof}[Proof of Lemma~\ref{lem:Hcal_gen}]
The proof is somehow similar to the one of Lemma~\ref{lem:Lcal_gen}. Fix $ t\geqslant 0$, $\eta\in \mathcal P_N$, denote
\[
\eta=\eta_{\x}=\frac{1}{N}\sum_{i=1}^N\delta_{x_i},\qquad (x_1,\cdots,x_N)\in E^N,
\]
and for each $i\in \llbracket1,N\rrbracket$, let $Y^i$ be a Markov process with initial condition $x_i$ and generator $L_{\eta}$. We claim that 
\begin{equation}\label{eq:R_expli_exp}
 R^{\rm exp}\br{\ph}(\eta_\x) = 
   \lim_{u \to 0}  \frac1u \frac{1}{N} \sum_{i = 1}^N  \mathbb E \br{ \ee^{ \frac{\delta_{N} \ph}{ {\delta_{N} \eta} }(\eta_{\x})(\delta_{Y^{i}_u} - \delta_{x_i}) } -  \ee^{ \frac{\delta \ph}{ {\delta \eta} }(\eta_{\x})(\delta_{Y^{i}_u} - \delta_{x_i})} } .
\end{equation}
Indeed, using the definition of the discrete derivative from Definition~\ref{def:discrete-derivative}, it holds:
\begin{multline*}
{\rm e}^{N\po \ph_t\po \eta + \frac{1}{N}(\delta_{Y^{i}_u}-\delta_{x_i}) \pf - \ph_t(\eta)\pf}- 1  = {\rm e}^{ \frac{\delta_{N} \ph_t}{ \delta_{N} \eta}(\eta)(\delta_{Y^{i}_u}-\delta_{x_i}) }- 1  \\= \frac{\delta \ph_t}{ \delta \eta}(\eta,Y^{i}_u) - \frac{\delta \ph_t}{ \delta \eta}(\eta,x_i) + \po {\rm e}^{\frac{\delta \ph_t}{ \delta \eta}(\eta,Y^{i}_u)-\frac{\delta \ph_t}{ \delta \eta}(\eta,x_i)} -1 - \po \frac{\delta \ph_t}{ \delta \eta}(\eta,Y^{i}_u)-\frac{\delta \ph_t}{ \delta \eta}(\eta,x_i) \pf \pf \\ + \po {\rm e}^{\frac{\delta_{N} \ph_t}{ \delta_{N} \eta}(\eta)(\delta_{Y^{i}_u}-\delta_{x})}- {\rm e}^{\frac{\delta \ph_t}{ \delta \eta}(\eta,Y^{i}_u)-\frac{\delta \ph_t}{ \delta \eta}(\eta,x)} \pf.
\end{multline*}
Under Assumption~\ref{assu:domain}, $x_i\mapsto \ee^{N\varphi\po\eta^{\x}\pf} \in \mathcal D\po L_\eta \pf$. Hence, we have that 
\begin{equation}\label{eq:proof-Hcal_Gamma-1}
\lim_{u\rightarrow 0}\frac{1}{u}\E\po{\rm e}^{N\po \ph_t\po \eta + \frac{1}{N}(\delta_{Y^{i}_u}-\delta_{x_i}) \pf - \ph_t(\eta)\pf}- 1\pf = {\rm e}^{-N\ph_t}L_{\eta}(x_i\mapsto {\rm e}^{N\ph_t}(\eta_{\x})).
\end{equation}
On the other hand it still holds that:
\begin{equation}\label{eq:proof-Hcal_Gamma-2}
\lim_{u\rightarrow 0}\frac{1}{u}\E\po \frac{\delta\ph_t}{\delta \eta} (\eta,Y^{i}_u) - \frac{\delta\ph_t}{\delta \eta} (\eta,x_i) \pf = L_{\eta}\po \frac{\delta \ph_t}{\delta \eta}(\eta,\cdot) \pf (x_i);
\end{equation}
and by definition of the exponential \textit{carr\'e-du-champs} from~\eqref{eq:exp-carre-du-champ}, we also have that
\begin{equation}\label{eq:proof-Hcal_Gamma-3}
\lim_{u\rightarrow 0}\frac{1}{u}\E\po {\rm e}^{\frac{\delta \ph_t}{ \delta \eta}(\eta,Y^{i}_u)-\frac{\delta \ph_t}{ \delta \eta}(\eta,x_i)} -1 - \po \frac{\delta \ph_t}{ \delta \eta}(\eta,Y^{i}_u)-\frac{\delta \ph_t}{ \delta \eta}(\eta,x_i) \pf \pf = \Gamma_{\eta}^{\rm exp}\po \frac{\delta \ph_t}{\delta \eta}(\eta,\cdot) \pf (\eta,x_i).
\end{equation}
Hence the claim~\eqref{eq:R_expli_exp}.

Now to upper bound the remainder term, write
\[
    {\rm e}^{\frac{\delta_{N} \ph_t}{ \delta_{N} \eta}(\eta)(\delta_{Y^{i}_u}-\delta_{x})}- {\rm e}^{\frac{\delta \ph_t}{ \delta \eta}(Y^{i}_u)-\frac{\delta \ph_t}{ \delta \eta}(x)} = \frac{1}{2N}\frac{\delta^{2}_N \ph_t}{ \delta_N\eta^{2} }(\eta)(\delta_{Y^{i}_u}-\delta_{x}) \kappa(\ph_t,x,Y^{i}_u),  
\]
where
\[
\kappa(\ph,x,x') = {\rm e}^{\frac{\delta \ph}{ \delta \eta}(x')-\frac{\delta \ph}{ \delta \eta}(x)} \frac{{\rm e}^{\frac{1}{2N}\frac{\delta_{N}^2 \ph}{ \delta_{N} \eta^2}(\delta_{x'}-\delta_{x})}-1 }{\frac{1}{2N}\frac{\delta_{N}^2 \ph}{ \delta_{N} \eta^2}(\delta_{x'}-\delta_{x})}.
\]
Lemma~\ref{lem:bounded-exp-norm-inf} and the inequality
\begin{equation*}
 0 \leqslant  \frac{ {\rm e}^{\theta} -1 }{\theta}\leqslant {\rm e}^{|\theta|},\qquad \forall\, \theta\in\R,
\end{equation*}
yield that for a constant $C$ depending only on $C_+,c_-$ of Condition~\eqref{eq:borne-semi-group}:
\[
0 \leqslant  \kappa(\ph,x,x') \leqslant {\rm e}^{\frac{\delta \ph}{ \delta \eta}(\eta,x')-\frac{\delta \ph}{ \delta \eta}(\eta,x)} {\rm e}^{\frac{1}{2N}\abs{\frac{\delta_{N}^2 \ph}{ \delta_{N} \eta^2}(\eta)(\delta_{x'}-\delta_{x})}} \leqslant {\rm e}^{C \norm{f}_\infty},
\]
so that, using the upper bound~\eqref{eq:proof-Lcal_gen-3-k1} from the proof of Lemma~\ref{lem:Lcal_gen}, we get
\begin{align*}
\limsup_{u\rightarrow 0} &\frac{1}{u}\left|\E\po {\rm e}^{\frac{\delta_{N} \ph_t}{ \delta_{N} \eta}(\eta)(\delta_{Y^{i}_u}-\delta_{x_i})}- {\rm e}^{\frac{\delta \ph_t}{ \delta \eta}(Y^{i}_u)-\frac{\delta \ph_t}{ \delta \eta}(x_i)} \pf  \right| \\&\leqslant \frac{1}{2N}{\rm e}^{C \norm{f}_\infty} \limsup_{u\rightarrow 0} \frac{1}{u} \left|\frac{\delta^2_N\ph_t}{\delta_N\eta^2}(\eta)(\delta_{Y^{i}_u} - \delta_{x_i})\right| \\&\leqslant 2\frac{{\rm e}^{C \norm{f}_\infty}}{N}\frac{1}{\po\eta^{(i)} \po Q^V_{T-t}(\one)\pf\pf^2}\\&\qquad \qquad \qquad \times \po \sqrt{\Gamma\po Q^V_{T-t}\po \overline f \pf \pf (x_i)}\sqrt{\Gamma\po Q^V_{T-t}(\one) \pf (x_i)} + \Phi_{T-t}(\eta)\po \overline f \pf \Gamma\po Q^V_{T-t}(\one) \pf (x_i) \pf.
\end{align*}
\end{proof}

\begin{proof}[Proof of Lemma~\ref{lem:bound-flat-derivative}]
As before, fix $ t\geqslant 0$, $\eta\in \mathcal P_N$ and denote
\[
\eta= \eta_{\x}=\frac{1}{N}\sum_{i=1}^N\delta_{x_i},\qquad (x_1,\cdots,x_N)\in E^N.
\]
For some $i\in \llbracket1,N\rrbracket$, let $Y$ be an independent Markov process with initial condition $x_i$ and generator $L_{\eta}$. Lemma~\ref{lem:bounded-exp-norm-inf} and the inequality
\[
0 \leqslant {\rm e}^{\theta} -1 - \theta \leqslant {\rm e}^{|\theta|}\frac{\theta^2}{2},\qquad \forall\, \theta\in\R,
\]
yield that
\begin{align}\label{eq:borne-exp-i-fixed}
    0\leqslant {\rm e}^{\frac{\delta \ph_t}{ \delta \eta}(Y_u)-\frac{\delta \ph_t}{ \delta \eta}(x_i)} -1 - \po \frac{\delta \ph_t}{ \delta \eta}(Y_u)-\frac{\delta \ph_t}{ \delta \eta}(x_i) \pf & \leqslant \frac{{\rm e}^{\abs{ \frac{\delta \ph_t}{ \delta \eta}(Y_u)-\frac{\delta \ph_t}{ \delta \eta}(x_i)}}}{2} \po \frac{\delta \ph_t}{ \delta \eta}(Y_u)-\frac{\delta \ph_t}{ \delta \eta}(x_i) \pf^2,\nonumber \\
    & \leqslant \frac{{\rm e}^{C \norm{f}_\infty}}{2} \po \frac{\delta \ph_t}{ \delta \eta}(Y_u)-\frac{\delta \ph_t}{ \delta \eta}(x_i) \pf^2
\end{align}
where the constant $C$ depends only on $C_+,c_-$ of Condition \eqref{eq:borne-semi-group}.
By definition of the exponential \textit{carr\'e-du-champ}, we then have that
\[
\lim_{t\rightarrow 0}\frac{1}{t}\E\po {\rm e}^{\frac{\delta \ph_t}{ \delta \eta}(Y_u)-\frac{\delta \ph_t}{ \delta \eta}(x_i)} -1 - \po \frac{\delta \ph_t}{ \delta \eta}(Y_u)-\frac{\delta \ph_t}{ \delta \eta}(x_i) \pf \pf = \Gamma^{\rm exp}_{\eta}\po \frac{\delta \ph_t}{ \delta \eta}\pf(x_i),
\]
and  Lemma~\ref{lem:carre-du-champ-as-derivatives} yields:
\[
\lim_{t\rightarrow 0}\frac{1}{t}\E\po \po \frac{\delta \ph_t}{ \delta \eta}(Y_u)-\frac{\delta \ph_t}{ \delta \eta}(x_i) \pf^2\pf = \Gamma_{L_\eta}\po \frac{\delta \ph_t}{ \delta \eta}\pf(x_i).
\]
Thus, taking the expectation in Equation~\eqref{eq:borne-exp-i-fixed}, dividing by $t$, letting $t$ goes to zero and integrating with respect to $\eta$ (that is averaging over particles), we get
\[
0 \leqslant \eta \Gamma^{\rm exp}_{\eta}\po \frac{\delta \ph_t}{ \delta \eta}\pf \leqslant \frac{{\rm e}^{C \norm{f}_\infty}}{2} \eta \Gamma_{L_\eta}\po \frac{\delta \ph_t}{ \delta \eta}\pf.
\]
Now, recall that Lemma~\ref{lem:derivative_flow} reads
\[
\frac{\delta \ph_t}{ \delta \eta}(x) = \frac{1}{\eta Q^V_{T-t}(\one)}  Q^V_{T-t}\po \overline f \pf(x) - \frac{\eta Q^V_{T-t}\po \overline f \pf}{\eta Q^V_{T-t}(\one)^2}  Q^V_{T-t}(\one)(x)
\]
and the simple identity$(a-b) \otimes (a-b)  = 2 a \otimes a + 2 b \otimes b - (a+b) \otimes (a+b)$ gives
\[
\eta \Gamma_{L_\eta}\po \frac{\delta \ph_t}{ \delta \eta}\pf \\\leqslant \frac{2}{\po\eta \po Q^V_{T-t}(\one) \pf\pf^2}\po \eta \Gamma_{L_\eta}\po Q^V_{T-t}\po \overline f \pf\pf + \Phi_{T-t}^2(\eta)\po \overline f \pf\eta \Gamma_{L_\eta}\po Q^V_{T-t}(\one) \pf \pf,
\]
which concludes the proof.
\end{proof}

\begin{proof}[Proof of Proposition~\ref{prop:mart-exp}]
Again, fix, $T>0$, and recall that for some $f\in \mathcal A$,
\[
\overline f = f - \Phi_T(\eta_0^N)(f).
\]
The evolution equation of the semi-group Equation~\eqref{eq:evolution2} yields
\begin{equation}\label{eq:evolution-exp}
\partial_t {\rm e}^{N \ph_t} = N {\rm e}^{N\ph_t}\partial_t\ph_{t} =-N{\rm e}^{N\ph_t} \eta  L_\eta \po \frac{\delta \ph_{t}}{\delta\eta} \pf.
\end{equation}
Equation~\eqref{eq:evolution-exp} and Lemma~\ref{lem:Hcal_gen} yield
\[
\int_0^t {\rm e}^{-N\ph_s} \po \partial_s+\mathcal L \pf ({\rm e}^{N\ph_s}) (\eta^N_s) \, \dd s = N\int_0^t \left[\eta_s^N\Gamma_{L_{\eta^N_s}}^{\rm exp}\po \frac{\delta \ph_s}{ \delta \eta}(\eta_s^N,\cdot) \pf  + \frac{1}{N} \eta^N_sR^{\rm exp}\br{ \ph_{s}}(\eta_s^N) \right]\dd s.
\]
We then conclude with the expression of the martingale  $M^{\rm exp,T}$ in~\eqref{eq:mart-exp} and the bounds from Lemma~\ref{lem:Hcal_gen} and Lemma~\ref{lem:bound-flat-derivative}.
\end{proof}

\subsection*{Acknowledgments}

L.J.  is supported by the grant n 200029-21991311 from the Swiss National Science Foundation. M.R. partially supported by ANR SINEQ, ANR-21-CE40-0006.

\bibliography{bibliographie,mathias}

@article{down1995exponential,
  title={Exponential and uniform ergodicity of Markov processes},
  author={Down, Douglas and Meyn, Sean P and Tweedie, Richard L},
  journal={The Annals of Probability},
  volume={23},
  number={4},
  pages={1671--1691},
  year={1995},
  publisher={Institute of Mathematical Statistics}
}

@article{BurHolMar00,
	author = {Krzysztof Burdzy and Robert Ho{\l}yst and Peter March},
	day = {01},
	doi = {10.1007/s002200000294},
	issn = {1432-0916},
	journal = {Comm. Math. Phys.},
	number = {3},
	pages = {679--703},
	title = {A {F}leming--{V}iot Particle Representation of the {D}irichlet {L}aplacian},
	url = {https://doi.org/10.1007/s002200000294},
	volume = {214},
	year = {2000}}

@Article{gradient-estimate,
 Author = {Kuwada, Kazumasa},
 Title = {Duality on gradient estimates and {Wasserstein} controls},
 FJournal = {Journal of Functional Analysis},
 Journal = {J. Funct. Anal.},
 ISSN = {0022-1236},
 Volume = {258},
 Number = {11},
 Pages = {3758--3774},
 Year = {2010},
 Language = {English},
 DOI = {10.1016/j.jfa.2010.01.010},
 zbMATH = {5708624},
 Zbl = {1194.53032}
}

@Article{Rousset06,
 Author = {Rousset, Mathias},
 Title = {On the control of an interacting particle estimation of {Schr{\"o}dinger} ground states},
 FJournal = {SIAM Journal on Mathematical Analysis},
 Journal = {SIAM J. Math. Anal.},
 ISSN = {0036-1410},
 Volume = {38},
 Number = {3},
 Pages = {824--844},
 Year = {2006},
 Language = {English},
 DOI = {10.1137/050640667},
 URL = {semanticscholar.org/paper/5e3d78e94279ae3fe0b4ebe1c6ed7e67bab67f82},
 zbMATH = {5155707},
 Zbl = {1174.60045}
}

@Article{couplage-non-convexe,
 Author = {Eberle, Andreas},
 Title = {Reflection couplings and contraction rates for diffusions},
 FJournal = {Probability Theory and Related Fields},
 Journal = {Probab. Theory Relat. Fields},
 ISSN = {0178-8051},
 Volume = {166},
 Number = {3-4},
 Pages = {851--886},
 Year = {2016},
 Language = {English},
 DOI = {10.1007/s00440-015-0673-1},
 zbMATH = {6657178},
 Zbl = {1367.60099}
}

@Article{conv-QSD,
 Author = {Champagnat, Nicolas and Villemonais, Denis},
 Title = {General criteria for the study of quasi-stationarity},
 FJournal = {Electronic Journal of Probability},
 Journal = {Electron. J. Probab.},
 ISSN = {1083-6489},
 Volume = {28},
 Pages = {84},
 Note = {Id/No 22},
 Year = {2023},
 Language = {English},
 DOI = {10.1214/22-EJP880},
 URL = {projecteuclid.org/journals/electronic-journal-of-probability/volume-28/issue-none/General-criteria-for-the-study-of-quasi-stationarity/10.1214/22-EJP880.full},
 zbMATH = {7707081},
 Zbl = {1533.60007}
}

@Article{conv-QSD-uniform,
 Author = {Champagnat, Nicolas and Villemonais, Denis},
 Title = {Lyapunov criteria for uniform convergence of conditional distributions of absorbed {Markov} processes},
 FJournal = {Stochastic Processes and their Applications},
 Journal = {Stochastic Processes Appl.},
 ISSN = {0304-4149},
 Volume = {135},
 Pages = {51--74},
 Year = {2021},
 Language = {English},
 DOI = {10.1016/j.spa.2020.12.005},
 zbMATH = {7339594},
 Zbl = {1469.60248}
}

@article{burdzy2000fleming,
  title={A Fleming--Viot Particle Representation of the Dirichlet Laplacian},
  author={Burdzy, Krzysztof and Ho{\l}yst, Robert and March, Peter},
  journal={Communications in Mathematical Physics},
  volume={214},
  number={3},
  pages={679--703},
  year={2000},
  publisher={Springer}
}

@article{bieniek2012non,
  title={Non-extinction of a Fleming-Viot particle model},
  author={Bieniek, Mariusz and Burdzy, Krzysztof and Finch, Sam},
  journal={Probability Theory and Related Fields},
  volume={153},
  number={1},
  pages={293--332},
  year={2012},
  publisher={Springer}
}

@article{Ferrari,
author = {Pablo Ferrari and Nevena Maric},
title = {{Quasi Stationary Distributions and Fleming-Viot Processes in Countable Spaces}},
volume = {12},
journal = {Electronic Journal of Probability},
number = {none},
publisher = {Institute of Mathematical Statistics and Bernoulli Society},
pages = {684 -- 702},
year = {2007},
doi = {10.1214/EJP.v12-415},
URL = {https://doi.org/10.1214/EJP.v12-415}
}

@article{tough2022fleming,
  title={The Fleming-Viot process with McKean-Vlasov dynamics},
  author={Tough, Oliver and Nolen, James},
  journal={Electronic Journal of Probability},
  volume={27},
  pages={1--72},
  year={2022},
  publisher={The Institute of Mathematical Statistics and the Bernoulli Society}
}

@article{villemonais2014general,
  title={General approximation method for the distribution of Markovprocesses conditioned not to be killed},
  author={Villemonais, Denis},
  journal={ESAIM: Probability and Statistics},
  volume={18},
  pages={441--467},
  year={2014},
  publisher={EDP Sciences}
}

@article{asselah2016fleming,
  title={Fleming--Viot selects the minimal quasi-stationary distribution: The Galton--Watson case},
  author={Asselah, Amine and Ferrari, Pablo A and Groisman, Pablo and Jonckheere, Matthieu},
    journal={Annales de l’Institut Henri Poincar{\'e}-Probabilit{\'e}s et Statistiques},
  volume={52},
  number={2},
  pages={647--668},
  year={2016}
}

@article{champagnat2021minimal,
  title={Convergence of the Fleming-Viot process toward the minimal quasi-stationary distribution},
  author={Champagnat, Nicolas and Villemonais, Denis},
  journal={ALEA},
  volume={18},
  pages={1--15},
  year={2021}
}

@article{tough2025selection,
  title={Selection principle for the Fleming--Viot particle system on the positive half-line with constant negative drift},
  author={Tough, Oliver},
  journal={Journal of the European Mathematical Society},
  year={2025}
}

@article{picard2002gradient,
  title={Gradient estimates for some diffusion semigroups},
  author={Picard, Jean},
  journal={Probability theory and related fields},
  volume={122},
  number={4},
  pages={593--612},
  year={2002},
  publisher={Springer}
}

@book{pinsky1995positive,
  title={Positive harmonic functions and diffusion},
  author={Pinsky, Ross G},
  volume={45},
  year={1995},
  publisher={Cambridge university press}
}

@book{kunita1990stochastic,
  title={Stochastic Flows and Stochastic Differential Equations},
  author={Kunita, Hiroshi},
  isbn={9780521599252},
  lccn={89070813},
  series={Cambridge Studies in Advanced Mathematics},
  url={https://books.google.fr/books?id=_S1RiCosqbMC},
  year={1990},
  publisher={Cambridge University Press}
}

@article{PRIOLA2006244,
author = {Priola, Enrico and Wang, Feng-Yu},
title = {Gradient estimates for diffusion semigroups with singular coefficients},
journal = {Journal of Functional Analysis},
volume = {236},
number = {1},
pages = {244-264},
year = {2006},
issn = {0022-1236}
}

@article{del2011concentration,
  title={Concentration inequalities for mean field particle models},
  journal={Annals of Applied Probability},
  author={Del Moral, Pierre and Rio, Emmanuel},
  year={2011}
}

@article{del2012concentration,
  title={On the concentration properties of interacting particle processes},
  author={Del Moral, Pierre and Hu, Peng and Wu, Liming and others},
  journal={Foundations and Trends in Machine Learning},
  volume={3},
  number={3--4},
  pages={225--389},
  year={2012},
  publisher={Now Publishers, Inc.}
}

@article{de2021backward,
  title={From the backward Kolmogorov PDE on the Wasserstein space to propagation of chaos for McKean-Vlasov SDEs},
  author={de Raynal, Paul-Eric Chaudru and Frikha, Noufel},
  journal={Journal de Math{\'e}matiques Pures et Appliqu{\'e}es},
  volume={156},
  pages={1--124},
  year={2021},
  publisher={Elsevier}
}

@book{protter2005stochastic,
  title={Stochastic differential equations},
  author={Protter, Philip E},
  year={2005},
  publisher={Springer}
}

@book{ethier2009markov,
  title={Markov processes: characterization and convergence},
  author={Ethier, Stewart N and Kurtz, Thomas G},
  year={2009},
  publisher={John Wiley \& Sons}
}

@article{mischler2013kac,
  title={Kac’s program in kinetic theory},
  author={Mischler, St{\'e}phane and Mouhot, Cl{\'e}ment},
  journal={Inventiones mathematicae},
  volume={193},
  pages={1--147},
  year={2013},
  publisher={Springer}
}

@book{DelMic00,
  title={Branching and interacting particle systems approximations of Feynman-Kac formulae with applications to non-linear filtering},
  author={Del Moral, Pierre and Miclo, Laurent},
  year={2000},
  publisher={Springer}
}

@article{cloez2022uniform,
  title={Uniform in time propagation of chaos for a Moran model},
  author={Cloez, Bertrand and Corujo, Josu{\'e}},
  journal={Stochastic Processes and their Applications},
  volume={154},
  pages={251--285},
  year={2022},
  publisher={Elsevier}
}

@article{angeli2021limit,
  title={Limit theorems for cloning algorithms},
  author={Angeli, Letizia and Grosskinsky, Stefan and Johansen, Adam M},
  journal={Stochastic Processes and their Applications},
  volume={138},
  pages={117--152},
  year={2021},
  publisher={Elsevier}
}

@article{journel2022convergence,
  title={Convergence of a particle approximation for the quasi-stationary distribution of a diffusion process: Uniform estimates in a compact soft case},
  author={Journel, Lucas and Monmarch{\'e}, Pierre},
  journal={ESAIM: Probability and Statistics},
  volume={26},
  pages={1--25},
  year={2022},
  publisher={EDP Sciences}
}

@book{del2004feynman,
  title={Feynman-kac formulae},
  author={Del Moral, Pierre},
  year={2004},
  publisher={Springer}
}
\bibliographystyle{plain}

\end{document}